\documentclass[11pt]{amsart}
\usepackage{geometry}                
\usepackage{graphicx}
\usepackage{amssymb}
\usepackage{epstopdf}
\newtheorem{theorem}{Theorem}

\newtheorem{corollary}[theorem]{Corollary}

\newtheorem{definition}[theorem]{Definition}

\newtheorem{lemma}[theorem]{Lemma}
\newtheorem{notation}[theorem]{Notation}

\newtheorem{proposition}[theorem]{Proposition}
\newtheorem{remark}[theorem]{Remark}

\title{{\em Local} Hochschild Homology of Hilbert-Schmidt Operators on Simplicial Spaces}
\author{Nicolae Teleman \\
Dipartimento di Scienze Matematiche, Universita' Politecnica delle Marche \\
E-mail:  teleman@dipmat.univpm.it}
\date{}                          

\begin{document}
\maketitle


\section{abstract}  
{\em Local} Hochschild, cyclic Homology and $K$-theory were introduced by N. Teleman in \cite{Teleman_arXiv} with the purpose of unifying different settings of the index theorem. This paper is one of the research topics announced in \cite{Teleman_arXiv}, \S 10. 
The definition of these new objects inserts
 the Alexander-Spanier idea for defining the co-homology \cite{Spanier} into the corresponding constructions. This is done by allowing only chains which have {\em small} support about the diagonal. This definition, applicable at least in the case of the Banach sub-algebras of the algebra of bounded operators on the Hilbert space of $L_{2}$-sections in vector bundles, differs from various constructions due to A. Connes \cite{Connes}, A. Connes, H.Moscovici \cite{Connes-Moscovici}, M. Puschnigg \cite{Puschnigg}, 
 J. Cuntz \cite{Cuntz}.
 \par
 In this paper we prove that the {\em local} Hochschild homology of the Banach algebra of Hilbert-Schmidt operators on any countable, locally finite homogeneous simplicial complex  $X$ is naturally isomorphic the Alexander-Spanier homology of the space $X$, Theorem 1. This result may be used to compute the {\em local} periodic cyclic homology of the algebra of Hilbert-Schmidt operators on such spaces $X$. The same result should hold in the case of the algebra of trace class operators $\mathit{L}^{1}$ as well as in the case of smoothing operators $\mathit{s} \subset \mathit{L}^{1}$.
 \par
 In addition, the tools we introduce in this paper should apply also for computing the {\em local} Hochschild and periodic cyclic homology of the Schatten class ideals $\mathit{L}^{p}$, at least for the other values $1 < p < 2 $.
 \par
 Parts of what is presented here were stated in author's lecture at the International Alexandroff Reading Conference, Moscow, 21-25 May 2012.  

\section{introduction}  
The main result of this paper concerns the computation of the {\em local} continuous Hochschild homology of the Banach algebra of Hilbert-Schmidt operators on countable, locally finite homogeneous simplicial complexes $X$. Theorem 1 states
 that it  is naturally isomorphic the Alexander-Spanier homology of the space $X$.  This result complements \cite{Teleman_arXiv} Proposition 26 which states that the {\em local} continuous Hochschild homology of the algebra of trace class operators on smooth manifolds is at least as big as the Alexander-Spanier homology of the manifold. While the proof of  \cite{Teleman_arXiv} Proposition 26 used the Connes-Moscovici \cite{Connes-Moscovici} Theorem 3.9, the treatment we present here is independent of it. 
\par
We recall that the need to consider {\em local} homological objects \cite{Teleman_arXiv} comes from at least two directions.
On the one side, -i) the Hochschild and cyclic homology, as well as the topological $K$-theory of the Banach algebra of bounded operators and various Schatten classes of compact operators on the Hilbert space of $L_{2}$ sections on a space $X$ is trivial,  see e.g. \cite{Connes}, \cite{Connes-Moscovici}, \cite{Cuntz};  on the other side, -ii) although the Alexander-Spanier homology appears naturally in these papers, see \cite{Connes-Moscovici}, its entrance into the theory does occur dually, in the co-homological context. 
\par
The triviality of the homologies in -i) is not surprising because the Banach algebras involved are independent of the space onto which they operate and therefore they do not {\em see} the space. This situation is similar to the phenomenon which occurs into the construction of the Alexander-Spanier co-homology \cite{Spanier} before imposing a control onto the supports of the chains (see \S 8.1 for more details). 
Our definition of {\em local} Hochschild homology inserts in its construction the Alexander-Spanier idea of control on the supports. Theorem 1, proves that building into the theory the control over the support of the chains allows one to obtain the right result.
\par
The paper is essentially self-contained. To facilitate the reading of this paper, the paper provides the basic necessary prerequisites. The description of the structure and techniques of the paper follow.
\par
Our computation of the Hochschild and {\em local} Hochschild homology of the algebra of Hilbert-Schmidt operators is based on a wavelet description of both.
\par
In \S 4 we gave some basic facts about Hilbert-Schmidt operators. To start our considerations we choose
an ortho-normal base in the Hilbert space of $\L_{2}$-functions on each maximal dimension simplex of the simplicial complex $X$. In the \S 10 we pass to analyse the {\em local} part of the argument. Here we see that having to consider smaller and smaller supports of the chains we are forced to consider finer and finer subdivisions of $X$ with the corresponding ortho-normal Hilbert bases, which explains the wavelets structure stated above.
\par
\S 5 recalls the basic definitions of the Hochschild homology. 
\par
\S 6 introduces the basic algebraic constructions of the paper. The continuous Hochschild complex over the algebra of Hilbert-Schmidt operators is decomposed in two sub-complexes (Proposition 7): the sub-complex $C^{0}_{\ast}(HS)$ generated by all chains which possess a {\em gap} in the kernel (Definition 4) and the {\em diagonal} sub-complex $C^{\Delta}_{\ast}(HS)$ generated by kernels without gaps.
\par
In \S 6.1 the operator $s$ is defined on the sub-complex $C^{0}_{\ast}(HS)$. Lemma 9 states that this sub-complex is acyclic. Therefore, the Hochschild homology of the algebra of Hilbert-Schmidt operators is the homology of the diagonal sub-complex $C^{\Delta}_{\ast}(HS)$.
\par
The elements of the diagonal sub-complex $C^{\Delta}_{\ast}(HS)$ have a simple description, see (14). Its elements present a {\em continuity} of the wavelets description (both, in terms of supports and elements of the ortho-normal base), in which a {\em trailing} phenomenon, both in terms of the supports and elements of the ortho-normal basis manifests. 
\par
In \S 6.2 we introduce the homotopy operator $S$ on the diagonal complex, Definition 10. It shows that the identity mapping is homotopic to an operator $\theta$. The operator $\theta$ replaces  (in the expression of the first tensor-factor of  the chain) any element of the ortho-normal base with a chosen element $I$ of the ortho-normal base, see Proposition 13. Each time the operator $\theta$ is applied, just one of the elements of the ortho-normal base is replaced with the fixed element $I$.  
After $p+1$ such modifications, any $p$-chain of the diagonal complex becomes  homotopic to a chain which uses only the chosen element $I$, on each maximal simplex of the space $X$,  see Proposition 27 i). Such elements form a sub-complex of the diagonal complex. We call it  {\em reduced diagonal complex}  and we denote it by  $C^{I}_{\ast}(HS)$.
\par
The homology of the reduced diagonal complex is analysed by means of a new homotopy operator,  $\tilde{S}$, defined on this the 
sub-complex, see formula (49). The operator $\tilde{S}$ is essentially the operator $S$ multiplied by a polynomial in the operator 
$\theta$. Theorem 28 states further that the homology of the reduced diagonal complex is isomorphic to the homology of the diagonal complex, and therefore it gives the Hochschild homology of the algebra of Hilbert-Schmidt operators.
 \par 
 We stress that, so far, we have not made any assumption on the supports of the chains. However, the homotopy operators $s$ and $S$ are local.  For this reason, when we will need to keep track of the supports of the chains, willing to keep them small,  we will still be able to use them. 
 \par
 In the same \S 6 we introduce the Notation 14, which will lead our steps toward the understanding of the parallelism between Hochschild homology of the algebra of Hilbert-Schmidt operators and the Alexander-Spanier homology. This would help us to understand, in particular, in topological terms, why the Hochschild homology, with no control on the supports of the chains, is trivial, Theorem 29. On the other side, this will help us to explain too, in \S 10,  why by considering {\em small} supports, both in the Hochschild homology complex and Alexander-Spanier complex, allows one to obtain the isomorphism between them. 
 \par
 In \S 8 we discuss Alexander-Spanier co-homology \S 8.1 and homology \S 8.2.
 \par
 In \S 8.3 we show that the the reduced diagonal complex, with no control on the supports, 
 is isomorphic to the Alexander-Spanier homology complex, with no control on the supports; therefore, they are trivial.
\par
In \S 7 we check that the operators we employ in our constructions keep us inside the continuous Hochschild complex.
\par
\S 9 we to take care of the supports, both in the Hochschild and Alexander-Spanier complexes.
Here we introduce a compatible simplicial filtration in both complexes. This filtration although interesting, is not used further
in this paper.
\par
Finally, in \S 10 we define carefully the {\em local} Hochschild homology of the algebra of Hilbert-Schmidt operators and we show that it is isomorphic to the Alexander-Spanier homology.  At this level the wavelets phenomenon appears more clearly. 
\par
The basic homotopies and isomorphisms discussed in \S 6-9 are local and therefore they pass to the {\em local} Hochschild and Alexander-Spanier homology complexes. The only difference, which appears by passing to the  {\em local} structures, occurs in the
homology of the Alexander-Spanier homology, which is trivial if no control is imposed on the supports and changes to the singular homology of the space, if the support controls are imposed. 
\par
The whole description of the {\em local} Hochschild homology we present here helps us to better understand how the 
{\em wavelets} pieces of the kernel of operators organise themselves to provide topological information.  
\par
The author thanks Jean-Paul Brasselet and Andr$\acute{e}$ Legrand for stimulant conversations.
\section{The Main Result.}     

\begin{theorem}
The {\em local} continuous Hochschild homology of the algebra of real, resp. complex, Hilbert-Schmidt operators on the countable, locally finite, homogeneous $n$-dimensional simplicial space $X$ is naturatelly isomorphic to the real, resp. complex, Alexander-Spanier homology of $X$.
\end{theorem}
\begin{proof}
The rest of this paper is devoted to the proof of this theorem.
\end{proof}

\section{Preliminaries and Notation}  

\subsection{The Space. Hilbert-Schmidt Kernels and Operators}  

Let $X$ be a connected, locally finite, countable simplicial set of dimension $n$.  Let $\Delta_{\alpha}$, $\alpha \in \Lambda$, denote all $n$-dimensional simplices of $X$. We assume that any simplex of $X$ is contained in an $n$-dimensional simplex of $X$; such a simplicial complex will be called {\em homogeneous}.
In particular, $X$ might be an $n$-dimensional pseudo-manifold or manifold.
We assume that each simplex $\Delta_{\alpha}$ is endowed with a Lebesque measure $\mu_{\Delta_{\alpha}}$.
\par
Let $\{ e_{\alpha}^{n}  \}_{n \in N}$ be an ortho-normal basis of $L_{2}$ real/complex valued functions on $\Delta_{\alpha}$. 
In the Introduction its  elements were referred to as wavelets. 
\par
Then the complex conjugates $\{ \bar{e}_{\alpha}^{n}  \}_{n \in N}$ form too an ortho-normal basis of $L_{2}(\Delta_{\alpha})$.
\par
A Hilbert-Schmidt kernel on $X$ is an $L_{2}$-function on $X \times X$. It is given by an $L_{2}$-convergent series
\begin{equation}
 K \;=\; \sum_{\alpha \beta, ij} K_{ij}^{\alpha \beta} \; (e_{\alpha}^{i} \times \bar{e}^{j}_{\beta} )
\end{equation}
with real/complex coefficients $K_{ij}^{\alpha \beta}$. Given the Hilbert-Schmidt kernel $K$, the decomposition (1) is unique.
\par
A Hilbert-Schmidt kernel of type $(e_{\alpha}^{i} \times \bar{e}^{j}_{\beta}) $ will be called {\em elementary}.
\par
Any Hilbert-Schmidt kernel $K$ defines a bounded Hilbert-Schmidt operator $\mathit{Op}(K): L_{2}(X) \rightarrow L_{2}(X) $
\begin{equation}
(\mathit{Op}(K)\phi) (x) := \int_{X} K(x,y) \; \phi (y) \; d\mu(y).
\end{equation}
\par
The composition of two elementary Hilbert-Schmidt operators is given by
\begin{equation}                  
\mathit{Op}K_{1} \circ \mathit{Op}K_{2} = \mathit{Op} K,
\end{equation}
where
\begin{equation}                  
         K (x, z) = \int_{X } K_{1} (x, y) . K_{2} (y, z) d \mu(y).
\end{equation} 
This kernel $K$ is by definition the composition of the kernels $K_{1}$, $K_{2}$, written $K = K_{1} \circ K_{2} $.
In other words,
\begin{equation}
\mathit{Op} ( K_{1} \circ K_{2}  ) =   \mathit{Op} ( K_{1} ) \circ  \mathit{Op} ( K_{2} ). 
\end{equation}
In particular, 

\begin{equation}                  
(e_{\alpha}^{i} \times \bar{e}^{j}_{\beta}) \circ (e_{\gamma}^{k} \times \bar{e}^{l}_{\eta} ) \;=\;  
             \delta^{jk}  \delta_{\beta \gamma} \; (e_{\alpha}^{i} \times \bar{e}^{l}_{\eta}),
\end{equation}
where $\delta^{jk}$ and $\delta_{\beta \gamma}$ are the Kronecker symbols.
\par
The Hilbert-Schmidt operators form an associative algebra denoted $\mathit{HS} (X)$.
\par
The algebra of Hilbert-Schmidt operators is the Schatten class $\mathit{L}^{2}$ of compact operators on the separable Hilbert space $H = L_{2}(X)$.


\section{Hochschild and {\em Local} Hochschild Homology of Hilbert-Schmidt Operators.}  
We recall the basic definitions regarding the Hochschild homology of associative algebras $\mathit{A}$.  
In this paper we compute the {\em local} Hochschild homology of the algebra $\mathit{A} = \mathit{HS} (X)$. 
{\em Local} Hochschild homology was defined by Teleman \cite{Teleman_arXiv}. It is the analogue of the Alexander-Spanier construction implanted into the Hochschild complex. This is done by considering only Hochschild chains which have small support about the main diagonal of the powers of the space $X$. 
\par
In the \S 5-7 of the paper we introduce certain algebraic constructions within the Hochschild complex of this algebra, ignoring the supports of the chains. It is important to stress here that all algebraic manipulations we introduce are
 {\em local} and therefore, they are well defined in the {\em local} Hochschild complex. These will enable us, at the end of this paper, in \S 10,  to complete the computation of  the {\em local}  Hochschild homology of the algebra of Hilbert-Schmidt operators and to connect it naturally with the Alexander-Spanier homology.
\par
The vector space of Hochschild $p$-chains of the algebra $\mathit{HS} (X)$ with values in itself is by definition
\begin{equation}    
C_{p} (\mathit{HS} (X)) \;=\; \otimes_{C}^{p+1} \mathit{HS} (X).
\end{equation}
\par
The Hochschild boundary $b_{(p)}: C_{p} (\mathit{HS} (X)) \longrightarrow C_{p-1} (\mathit{HS} (X))$ is 
\begin{equation} 
b_{(p)} = \sum_{k=0}^{k=p-1}  b_{(p)k} +  b_{(p)p}     
\end{equation}
where
\begin{equation}
b_{(p)k} = (-1)^{k} \partial_{k}^{H} ,    \hspace{0.2cm}  and  \hspace{0.3cm} b_{(p)p} = (-1)^{p} \partial_{p}^{H}    
\end{equation}
with
\begin{equation}
\partial_{(p)k}^{H}  (K_{0} \otimes_{C}  K_{1}  \otimes_{C} .... \otimes_{C} K_{p}  )  \;=\;           
        K_{0} \otimes_{C} ... \otimes_{C} K_{k-1} \otimes_{C} (K_{k} \circ  K_{k+1}) \otimes_{C} ... \otimes_{C} K_{p} 
\end{equation}
and
\begin{equation}     
\partial_{(p)p} ^{H} (K_{0} \otimes_{C}  K_{1}  \otimes_{C} .... \otimes_{C} K_{p}  )  \;=\;     
         ( K_{p} \circ K_{0} ) \otimes_{C}  K_{1}  \otimes_{C} .... \otimes_{C} K_{p-1} .
\end{equation}
The operator $b_{(p)k}^{H}$ is called {\em  Hochschild face operator of order} $k$.
When no confusion occurs, the index $(p)$, indicating the degree of chains, could be omitted.
\par

\begin{definition}   
Let $K  = K_{0} \otimes_{C}  K_{1}  \otimes_{C} .... \otimes_{C} K_{p} $ where each of the factors $K_{i}$ is elementary.
Then $K$ is called {\em elementary chain}.
\par
When the ground algebra $\mathit{A}$ is a locally convex topological algebra, it is customary to replace the algebraic tensor products by projective tensor products, see Connes \cite{Connes}. The homology of the corresponding completed complex is called {\em continuous} Hochschild homology of the algebra $\mathit{A}$. In this paper we consider the completion $\bar{C}_{p}(\mathit{HS})$ defined below.
\end{definition}
\begin{definition}
Let $\bar{C}_{p}(\mathit{HS})$ denote the space of $L_{2}$-functions on $(X \times X)^{p+1}$. 
\par
\end{definition}
\par
All considerations made in the sequel refer to the computation of the homology of the completed complex 
$ \{  \bar{C}_{\ast}(\mathit{HS}), b \}_{\ast}$.
\par
The elementary chains form an ortho-normal basis of the Hilbert space  $\bar{C}_{p}(\mathit{HS})$.
To simplify the notation we agree to denote  $\bar{C}_{p}(\mathit{HS})$ by $C_{p}(\mathit{HS})$.

\section{Algebraic Constructions}    

\begin{definition}    
Let $K$ be an elementary $p$-chain and let $k$ any index $0 \leq k \leq p$. We say that $K$ has a $k$-{\em gap} provided 
$\partial_{k}^{H} K = 0$.
\end{definition}
\begin{remark}   
The elementary chain $K$ has a $k$-gap (for $k = p$ we intend $p+1$ to be $0$) provided the consecutive elementary factors 
\begin{equation}
(e^{i_{k}}_{\alpha_{k}}  \times \bar{e}^{j_{k}}_{\beta_{k}} )
                                            \otimes_{C} 
(e^{i_{k+1}}_{\alpha_{k+1}}  \times \bar{e}^{j_{k+1}}_{\beta_{k+1}} )
\end{equation}
either have different supports   $\beta_{k} \neq \alpha_{k+1}$, or they involve different elements of the Hilbert space basis, $e^{j_{k}} \neq e^{i_{k+1}}$, or both.
\end{remark} 
\begin{definition}   
\par
i) Let $C^{0}_{p}(\mathit{HS} (X))  \subset \bar{C}_{p}(\mathit{HS} (X)) $ be the vector subspace generated by those elementary chains $K$ which contain at least one gap.  
\par
By definition,   
\begin{equation}     
C^{0}_{0}(\mathit{HS} (X))    = \sum_{
(i_{0}, \alpha_{0}) \neq (j_{0}, \beta_{0}) 
}  \;
K_{i_{0}j_{0}}^{\alpha_{0} \beta_{0} }  \;
e^{i_{0}}_{\alpha_{0}} \times \bar{e}^{j_{0}}_{\beta_{0}}
\end{equation}    
\par
ii)  $\{ C^{0}_{p}(\mathit{HS} (X)), \; b\}_{0 \leq p}$  is a sub-complex of the Hochschild complex.
\par
iii) Let $C^{\Delta}_{p}(\mathit{HS} (X))  \subset \bar{C}_{p}(\mathit{HS}_{0} (X)) $ be the vector subspace generated by those elementary chains $K$ which posses no gap.
\par
Any chain belonging to  $C^{\Delta}_{p}(\mathit{HS} (X)) $ is the sum of an $L_{2}$-convergent series of elementary chains 
\begin{equation}     
 K \;=\; \sum_{\alpha_{k}, i_{k}} K_{i_{0}, ..., i_{p}}^{\alpha_{i_{0}, ..., i_{p}}} \; (e_{\alpha_{0}}^{i_{0}} \times \bar{e}^{i_{1}}_{\alpha_{1}} )
 \otimes_{C} 
 (e_{\alpha_{1}}^{i_{1}} \times \bar{e}^{i_{2}}_{\alpha_{2}} )  \otimes_{C}  ... \otimes_{C}  (e_{\alpha_{p}}^{i_{p}} \times \bar{e}^{i_{0}}_{\alpha_{0}} ), 
 \end{equation}      
 with
 \begin{equation}
 \sum_{\alpha_{k}, i_{k}} | K_{i_{0}, ..., i_{p}}^{\alpha_{i_{0}, ..., i_{p}}} |^{2} < \infty.
 \end{equation}
\par
By definition,   
\begin{equation}          
C^{\Delta}_{0}(\mathit{HS} (X))     
 = \sum_{
(i_{0}, \alpha_{0}) 
}  \;
K_{i_{0}i_{0}}^{\alpha_{0} \alpha_{0} }  \;
e^{i_{0}}_{\alpha_{0}} \times \bar{e}^{i_{0}}_{\alpha_{0}}.
\end{equation}        
iv) $C^{\Delta}_{p}(\mathit{HS} (X)), \; b\}_{0 \leq p}$ is a complex. This complex is called {\em diagonal complex.} 
\end{definition}
Concerning ii), we observe that if a Hochschild boundary face $\partial_{(p)k}^{H}$ acts on a gap, the result is the zero chain; if the Hochschild boundary does not involve the gap, the gap survives, which shows that
  $C^{0}_{p}(\mathit{HS} (X)), \; b\}_{0 \leq p}$ is indeed a homology complex.
\par
Concerning iv), formula (6) shows that all Hochschild boundary faces keep the structure of the chains (14) unaltered, which shows that
 $C^{\Delta}_{p}(\mathit{HS} (X)), \; b\}_{0 \leq p}$ is a homology complex too.
\par
\begin{proposition}    
The Hochschild complex $\{  \bar{C}_{\ast} (\mathit{HS} (X)), \; b  \}_{\ast}$ decomposes into a direct sum of Hochschild sub-complexes 
\begin{equation}
\{ \bar{C}_{\ast} (\mathit{HS} (X)), \; b  \}_{\ast} = 
 \{  C^{0}_{\ast}(\mathit{HS} (X)), \; b  \}_{\ast}  \oplus    \{  C^{\Delta}_{\ast}(\mathit{HS} (X)), \; b  \}_{\ast}
\end{equation}
\end{proposition}    
\begin{proof}  
Obvious.
\end{proof}    
\subsection{Homotopy Operator $s$. The splitting}   
\par
\begin{proposition}      
i) The complex $\{  C^{0}_{\ast}(\mathit{HS} (X)), \; b  \}_{\ast} $  is acyclic.  
\par
ii) The inclusion of the diagonal sub-complex  $ \{  C^{\Delta}_{\ast}(\mathit{HS} (X)), \; b  \}_{\ast}$   in the complex  
$\{ C_{\ast} (\mathit{HS} (X)), \; b  \}_{\ast} $ induces isomorphism in homology.  
\par
iii) Therefore, the homology of the diagonal sub-complex $ \{  C^{\Delta}_{\ast}(\mathit{HS} (X)), \; b  \}_{\ast}$ 
is the continuous Hochschild homology of the algebra of Hilbert-Schmidt operators.
\end{proposition}   
\par
In view of Proposition 8, only part i) needs to be proven.
\par
The proof of i) is based on a homotopy operator $s := \{ s_{(p)} \}$,    
\begin{equation}      
s_{(p)}:   C^{0}_{p}(\mathit{HS} (X)) \longrightarrow  C^{0}_{p+1}(\mathit{HS} (X))
\end{equation}     
defined below. 
\par
The construction of the homotopy operator $s_{(p)}$ keeps track of the {\em first} gap present in each elementary monomial. 
For, suppose $K$ is an elementary $p$-chain which possesses an $r$-gap, with $r$ minimal, $0 \leq r \leq p$ 
\begin{equation}      
K =  (e^{i_{0}}_{\alpha_{0}}  \times \bar{e}^{i_{1}}_{\alpha_{1}} )  
                    \otimes_{\mathit{C}}                 
(e^{i_{1}}_{\alpha_{1}}  \times \bar{e}^{i_{2}}_{\alpha_{2}} )
       \otimes_{\mathit{C}}     .....    \otimes_{\mathit{C}}  
 (e^{i_{r}}_{\alpha_{r}}  \times \bar{e}^{j_{r}}_{\beta_{r}} )
                   \otimes_{\mathit{C}}  
(e^{i_{r+1}}_{\alpha_{r+1}}  \times \bar{e}^{j_{r+1}}_{\beta_{r+1}} ) 
                            \otimes_{\mathit{C}}  ...   \otimes_{\mathit{C}}                            
 (e^{i_{p}}_{\alpha_{p}}  \times \bar{e}^{j_{p}}_{\beta_{p}} ).
\end{equation}   
The operator $s_{(p)} (K)$ is defined by inserting the factor $ (e^{j_{r}}_{\beta_{r}}  \times \bar{e}^{j_{r}}_{\beta_{r}} )$
\begin{equation*}     
s_{(p)} (K) :=     
\end{equation*}
\begin{equation}
= (-1)^{r}
(e^{i_{0}}_{\alpha_{0}}  \times \bar{e}^{j_{0}}_{\beta_{0}} )  \otimes_{\mathit{C}} ... 
\otimes_{\mathit{C}}  (e^{i_{r}}_{\alpha_{r}}  \times \bar{e}^{j_{r}}_{\beta_{r}} ) 
\otimes_{\mathit{C}}  (e^{j_{r}}_{\beta_{r}}  \times \bar{e}^{j_{r}}_{\beta_{r}} )
\otimes_{\mathit{C}}  (e^{i_{r+1}}_{\alpha_{r+1}}  \times \bar{e}^{j_{r+1}}_{\beta_{r+1}} ) \otimes_{C}... 
\otimes_{\mathit{C}}  (e^{i_{p}}_{\alpha_{p}}  \times \bar{e}^{j_{p}}_{\beta_{p}} ).
\end{equation}
If $K$ is an elementary chain of degree $0$, then the homotopy  $s_{(0)}$ is defined by
\begin{equation}      
s_{(0)} (e^{i_{0}}_{\alpha_{0}} \times \bar{e}^{j_{0}}_{\beta_{0}}) := (e^{i_{0}}_{\alpha_{0}} \times \bar{e}^{j_{0}}_{\beta_{0}}) 
\otimes_{\mathit{C}}
(e^{j_{0}}_{\beta_{0}} \times \bar{e}^{j_{0}}_{\beta_{0}}). 
\end{equation}
\par
We postpone to the \S 7.3. the checking that the operators $s_{(p)}$ are well defined on the spaces $C_{p}^{0}(HS(X))$.

\begin{lemma}        
The operators $s$ satisfy 
\begin{equation}
(\;b \; s_{(p)}  + s_{(p-1)}\; b \;) (K) = K
\end{equation}
on $C^{0}_{\ast}(HS)$.
\end{lemma}   
\begin{proof}   
One has
\[
b \;s_{(p)} (K) = \sum_{k=0}^{k=r-1} b_{(p+1)k} (s_{(p)}K) +  b_{(p+1)r} (s_{(p)}K)  + b_{(p+1)r+1} (s_{(p)}K) + \sum_{k=r+2}^{k=p+1} b_{(p+1)k} (s_{(p)}K) =
\]
\[
 = \sum_{k=0}^{k=r-1} b_{(p+1)k} (s_{(p)}K) +  K + 0 + \sum_{k=r+2}^{k=p+1} b_{(p+1)k} (s_{(p)}K).
\]
Therefore
\begin{equation}               
b \; s_{(p)} (K) = \sum_{k=0}^{k=r-1} b_{(p+1)k} (s_{(p)}K) +  K + \sum_{k=r+2}^{k=p+1} b_{(p+1)k} (s_{(p)}K).
\end{equation}
On the other hand, we split the boundary $bK$ in three parts and we apply the operator $s_{(p-1)}$
\begin{equation*}        
s_{(p-1)}\; b(K) = \sum_{k=0}^{k=r-1}  s_{(p-1)}b_{(p)k} K  +     s_{(p-1)} b_{(p)r} K  +    \sum_{k=r+1}^{k=p}  s_{(p-1)} b_{(p)k} K =
\end{equation*}
\begin{equation}         
        = \sum_{k=0}^{k=r-1}  s_{(p-1)}b_{(p)k} K    +    \sum_{k=r+1}^{k=p}  s_{(p-1)} b_{(p)k} K.
\end{equation}
A direct check shows that both the first and the last terms from equations (23) and (24) cancel out, respectively.

Here, in particular, we intend to check the formula (22) when the first gap of $K_{p}$ is a $p$-gap 
\begin{equation}
K =  (e^{i_{0}}_{\alpha_{0}}  \times \bar{e}^{i_{1}}_{\alpha_{1}} )  
                    \otimes_{\mathit{C}}                 
(e^{i_{1}}_{\alpha_{1}}  \times \bar{e}^{i_{2}}_{\alpha_{2}} )
       \otimes_{\mathit{C}}     .....    \otimes_{\mathit{C}}  
 (e^{i_{p-1}}_{\alpha_{p-1}}  \times \bar{e}^{i_{p}}_{\alpha_{p}} ) 
                   \otimes_{\mathit{C}}  
(e^{i_{p}}_{\alpha_{p}}  \times \bar{e}^{j_{p}}_{\beta_{p}} )
\end{equation}
with $      (\beta_{p}, j_{p}) \neq    (\alpha_{0}, i_{0})$.
Then
\begin{multline}
b s_{(p)} (K) =  \\.
= (-1)^{p} \{ \sum_{k=0}^{k=p-1} b_{(p+1)k} 
(e^{i_{0}}_{\alpha_{0}}  \times \bar{e}^{i_{1}}_{\alpha_{1}} )  
                    \otimes_{\mathit{C}}                 
(e^{i_{1}}_{\alpha_{1}}  \times \bar{e}^{i_{2}}_{\alpha_{2}} )
       \otimes_{\mathit{C}}     .....    \otimes_{\mathit{C}}  
 (e^{i_{p-1}}_{\alpha_{p-1}}  \times \bar{e}^{i_{p}}_{\alpha_{p}} ) 
                   \otimes_{\mathit{C}}  
(e^{i_{p}}_{\alpha_{p}}  \times \bar{e}^{j_{p}}_{\beta_{p}} )  \}
                    \otimes_{\mathit{C}}  
        (e^{j_{p}}_{\beta_{p}}  \times \bar{e}^{j_{p}}_{\beta_{p}} ) + \\
 + (-1)^{p} (-1)^{p}  \partial_{(p+1)p}     
\{ (e^{i_{0}}_{\alpha_{0}}  \times \bar{e}^{i_{1}}_{\alpha_{1}} )  
                    \otimes_{\mathit{C}}                 
(e^{i_{1}}_{\alpha_{1}}  \times \bar{e}^{i_{2}}_{\alpha_{2}} )
       \otimes_{\mathit{C}}     .....    \otimes_{\mathit{C}}  
 (e^{i_{p-1}}_{\alpha_{p-1}}  \times \bar{e}^{i_{p}}_{\alpha_{p}} ) 
                   \otimes_{\mathit{C}}  
(e^{i_{p}}_{\alpha_{p}}  \times \bar{e}^{j_{p}}_{\beta_{p}} )  
                    \otimes_{\mathit{C}}  
        (e^{j_{p}}_{\beta_{p}}  \times \bar{e}^{j_{p}}_{\beta_{p}} ) \} + \\
 + (-1)^{p+1} (-1)^{p}  \partial_{(p+1)p+1}     
\{ (e^{i_{0}}_{\alpha_{0}}  \times \bar{e}^{i_{1}}_{\alpha_{1}} )  
                    \otimes_{\mathit{C}}                 
(e^{i_{1}}_{\alpha_{1}}  \times \bar{e}^{i_{2}}_{\alpha_{2}} )
       \otimes_{\mathit{C}}     .....    \otimes_{\mathit{C}}  
 (e^{i_{p-1}}_{\alpha_{p-1}}  \times \bar{e}^{i_{p}}_{\alpha_{p}} ) 
                   \otimes_{\mathit{C}}  
(e^{i_{p}}_{\alpha_{p}}  \times \bar{e}^{j_{p}}_{\beta_{p}} )  
                    \otimes_{\mathit{C}}  
(e^{j_{p}}_{\beta_{p}}  \times \bar{e}^{j_{p}}_{\beta_{p}} ) \} = \\   
=  (-1)^{p}  \sum_{k=0}^{k=p-1} b_{(p+1)k} 
(e^{i_{0}}_{\alpha_{0}}  \times \bar{e}^{i_{1}}_{\alpha_{1}} )  
                    \otimes_{\mathit{C}}                 
(e^{i_{1}}_{\alpha_{1}}  \times \bar{e}^{i_{2}}_{\alpha_{2}} )
       \otimes_{\mathit{C}}     .....    \otimes_{\mathit{C}}  
 (e^{i_{p-1}}_{\alpha_{p-1}}  \times \bar{e}^{i_{p}}_{\alpha_{p}} ) 
                   \otimes_{\mathit{C}}  
(e^{i_{p}}_{\alpha_{p}}  \times \bar{e}^{j_{p}}_{\beta_{p}} )  \}
                    \otimes_{\mathit{C}}  
        (e^{j_{p}}_{\beta_{p}}  \times \bar{e}^{j_{p}}_{\beta_{p}} ) + K.
\end{multline}
On the other hand,
\begin{equation}
s_{(p-1)} b (K) = s_{(p-1)} \{ \sum_{k=0}^{k=p-1} b_{(p)k} (K) + b_{(p)p} (K) \} = 
(-1)^{p-1}  \sum_{k=0}^{k=p-1} b_{(p)k} (K) \otimes_{\mathit{C}}  (e^{j_{p}}_{\beta_{p}}  \times \bar{e}^{j_{p}}_{\beta_{p}} ). 
\end{equation} 
Summing up formulas (26) and (27) shows that the identity (22) holds also for $p$-elementary chains whose first gap is a $p$-gap.
\par
Finally, we verify the homotopy formula (22) for chains of degree zero. For such chains, one has
\begin{equation}
(bs_{(0)} + s b) (e^{i_{0}}_{\alpha_{0}}  \times \bar{e}^{j_{0}}_{\beta_{0}} ) = b  s_{(0)} (e^{i_{0}}_{\alpha_{0}}  \times \bar{e}^{j_{0}}_{\beta_{0}} ) =
b(   \;   (e^{i_{0}}_{\alpha_{0}}  \times \bar{e}^{j_{0}}_{\beta_{0}} )         \otimes_{\mathit{C}}
 (e^{j_{0}}_{\beta_{0}}  \times \bar{e}^{j_{0}}_{\beta_{0}} )  \;)  =
 \end{equation}
 \begin{equation*}
= (e^{i_{0}}_{\alpha_{0}}  \times \bar{e}^{j_{0}}_{\beta_{0}} )   - 
  (e^{j_{0}}_{\beta_{0}}  \times \bar{e}^{j_{0}}_{\beta_{0}} )  \;) \circ  (e^{i_{0}}_{\alpha_{0}}  \times \bar{e}^{j_{0}}_{\beta_{0}} ) =  
 \end{equation*} 
 \begin{equation*}
=  (e^{i_{0}}_{\alpha_{0}}  \times \bar{e}^{j_{0}}_{\beta_{0}} )   - 
  \delta^{j_{0}i_{0}}    \delta_{\beta_{0} \alpha_{0}}
  ( e^{j_{0}} _{\alpha_{0}} \times \bar{e}^{j_{0}}_{\beta_{0}} ) =
    (e^{i_{0}}_{\alpha_{0}}  \times \bar{e} ^{j_{0}}_{\beta_{0}} )  
\end{equation*}
because $\delta^{j_{0}i_{0}}    \delta_{\beta_{0} \alpha_{0}} = 0$, which proves the desired formula.
 \par
These computations complete the proof of Lemma 9. This completes the proof of Proposition 8. -i).
\end{proof}    
Part -ii) of Proposition 8 follows from part -i) along with Proposition 7. This completes the proof of Proposition 8.

\subsection{Homotopy Operator  $S$}   
For each  $n$-simplex $\Delta_{\alpha}$ we choose and fix an element $e^{ I_{\alpha} }$ of the corresponding ortho-normal Hilbert basis. To simplify the notation we denote $e^{ I_{\alpha} }$  by $ I $. 
\begin{definition}   
Let $S_{(p)}:    C^{\Delta}_{p}(\mathit{HS} (X)) \longrightarrow C^{\Delta}_{p+1}(\mathit{HS} (X)) $ be defined on elementary chains by the formula
\begin{multline}    
S_{(p)} 
[(e_{\alpha_{0}}^{i_{0}} \times \bar{e}^{i_{1}}_{\alpha_{1}} )
 \otimes_{C} 
 (e_{\alpha_{1}}^{i_{1}} \times \bar{e}^{i_{2}}_{\alpha_{2}} )  
 \otimes_{C}  ... \otimes_{C}  
 (e_{\alpha_{p}}^{i_{p}} \times \bar{e}^{i_{0}}_{\alpha_{0}} )] := \\
:=  (e_{\alpha_{0}}^{i_{0}} \times \bar{e}^{I}_{\alpha_{0}} )   
  \otimes_{C} 
  (e_{\alpha_{0}}^{I} \times \bar{e}^{i_{1}}_{\alpha_{1}} )
 \otimes_{C} 
 (e_{\alpha_{1}}^{i_{1}} \times \bar{e}^{i_{2}}_{\alpha_{2}} )  \otimes_{C}  ... \otimes_{C}  (e_{\alpha_{p}}^{i_{p}} \times \bar{e}^{i_{0}}_{\alpha_{0}} ), 
\end{multline}
i.e. $S$ is defined by inserting the factor $\bar{e}^{I}_{\alpha_{0}} \otimes_{C} e^{I}_{\alpha_{0}}  $  into the expression of $K$.
\end{definition}   

\subsubsection{The Homology of the sub-complex $\{  C^{\Delta}_{p}(\mathit{HS} (X)), b' \}$. } 
\begin{proposition}    
The operator $S$ satisfies
\begin{equation}    
b' \; S + S \; b' = Id.
\end{equation}   
\end{proposition}    

\begin{proof}   
We compute first $b' S_{(p)} (K)$. One has
\begin{multline}
b' S_{(p)}  (K) =              
b' S_{(p)} 
[(e^{i_{0}} _{\alpha_{0}}\times \bar{e}^{i_{1}}_{\alpha_{1}} )
 \otimes_{C} 
 (e^{i_{1}} _{\alpha_{1}} \times \bar{e}^{i_{2}}_{\alpha_{2}} )  
 \otimes_{C}  ... \otimes_{C}  
 (e^{i_{p}}_{\alpha_{p}} \times \bar{e}^{i_{0}}_{\alpha_{0}} )] := \\
:=  b' [ (e^{i_{0}} _{\alpha_{0}}\times \bar{e}^{I}_{\alpha_{0}} )   
  \otimes_{C} 
  (e^{I}_{\alpha_{0}} \times \bar{e}^{i_{1}}_{\alpha_{1}} )
 \otimes_{C} 
 (e_{\alpha_{1}}^{i_{1}} \times \bar{e}^{i_{2}}_{\alpha_{2}} )  \otimes_{C}  ... \otimes_{C}  (e_{\alpha_{p}}^{i_{p}} \times \bar{e}^{i_{0}}_{\alpha_{0}} )] =  \\
=  (e^{i_{0}}_{\alpha_{0}} \times \bar{e}^{i_{1}}_{\alpha_{1}} )
\otimes_{C} 
 (e^{i_{1}}_{\alpha_{1}} \times \bar{e}^{i_{2}}_{\alpha_{2}} )  
 \otimes_{C}  ... \otimes_{C}
(e^{i_{p}}_{\alpha_{p}} \times \bar{e}^{i_{0}}_{\alpha_{0}} )   \\
 -  (e^{i_{0}}_{\alpha_{0}} \times \bar{e}^{I}_{\alpha_{0}} )  
  \otimes_{C}
 b' [  (e^{I}_{\alpha_{0}} \times \bar{e}^{i_{1}}_{\alpha_{1}} )
 \otimes_{C} 
 (e_{\alpha_{1}}^{i_{1}} \times \bar{e}^{i_{2}}_{\alpha_{2}} )  \otimes_{C}  ... \otimes_{C}  (e_{\alpha_{p}}^{i_{p}} \times \bar{e}^{i_{0}}_{\alpha_{0}} )] 
 = \\
 = K   
 -  (e^{i_{0}}_{\alpha_{0}} \times \bar{e}^{I}_{\alpha_{0}} )  
  \otimes_{C}
 b' [  (e^{I}_{\alpha_{0}} \times \bar{e}^{i_{1}}_{\alpha_{1}} )
 \otimes_{C} 
 (e_{\alpha_{1}}^{i_{1}} \times \bar{e}^{i_{2}}_{\alpha_{2}} )  \otimes_{C}  ... \otimes_{C}  (e_{\alpha_{p}}^{i_{p}} \times \bar{e}^{i_{0}}_{\alpha_{0}} )]. 
\end{multline}     
On the other hand
\begin{multline}              
S_{(p-1)} b' (K) = \\
= S_{(p-1)} [   
 (e^{i_{0}}_{\alpha_{0}} \times \bar{e}^{i_{2}}_{\alpha_{2}} )  
 \otimes_{C} 
  (e^{i_{2}}_{\alpha_{2}} \times \bar{e}^{i_{3}}_{\alpha_{3}} )  
\otimes_{C}  ... \otimes_{C}
(e^{i_{p}}_{\alpha_{p}} \times \bar{e}^{i_{0}}_{\alpha_{0}} )  \\
- (e^{i_{0}}_{\alpha_{0}} \times \bar{e}^{i_{1}}_{\alpha_{1}} )
\otimes_{C} 
 (e^{i_{1}}_{\alpha_{1}} \times \bar{e}^{i_{3}}_{\alpha_{3}} )  
 \otimes_{C}  ... \otimes_{C}
(e^{i_{p}}_{\alpha_{p}} \times \bar{e}^{i_{0}}_{\alpha_{0}} ) 
+ ....  \\
+(-1)^{k} (e^{i_{0}}_{\alpha_{0}} \times \bar{e}^{i_{1}}_{\alpha_{1}} )
\otimes_{C} 
 (e^{i_{1}}_{\alpha_{1}} \times \bar{e}^{i_{2}}_{\alpha_{2}} )  
 \otimes_{C}  ... \otimes_{C}
  (e^{i_{k}}_{\alpha_{k}} \times \bar{e}^{i_{k+2}}_{\alpha_{k+2}} )  
\otimes_{C}  ... \otimes_{C}
(e^{i_{p}}_{\alpha_{p}} \times \bar{e}^{i_{0}}_{\alpha_{0}} )  
+  ....  \\
+ (-1)^{p-1}
(e^{i_{0}}_{\alpha_{0}} \times \bar{e}^{i_{1}}_{\alpha_{1}} )
\otimes_{C} 
 (e^{i_{1}}_{\alpha_{1}} \times \bar{e}^{i_{2}}_{\alpha_{2}} )  
 \otimes_{C}  ... \otimes_{C}
(e^{i_{p-1}}_{\alpha_{p-1}} \times \bar{e}^{i_{0}}_{\alpha_{0}} ) ] = \\
= (e^{i_{0}}_{\alpha_{0}} \times \bar{e}^{I}_{\alpha_{0}} )  \otimes_{C}  (e^{I}_{\alpha_{0}} \times \bar{e}^{i_{2}}_{\alpha_{2}} ) 
 \otimes_{C} 
  (e^{i_{2}}_{\alpha_{2}} \times \bar{e}^{i_{3}}_{\alpha_{3}} )  
\otimes_{C}  ... \otimes_{C}
(e^{i_{p}}_{\alpha_{p}} \times \bar{e}^{i_{0}}_{\alpha_{0}} )  \\
- (e^{i_{0}}_{\alpha_{0}} \times \bar{e}^{I}_{\alpha_{0}} ) \otimes_{C}  (e^{I}_{\alpha_{0}} \times \bar{e}^{i_{1}}_{\alpha_{1}} ) 
\otimes_{C} 
 (e^{i_{1}}_{\alpha_{1}} \times \bar{e}^{i_{3}}_{\alpha_{3}} )  
 \otimes_{C}  ... \otimes_{C}
(e^{i_{p}}_{\alpha_{p}} \times \bar{e}^{i_{0}}_{\alpha_{0}} ) 
+ ....  \\
+(-1)^{k} (e^{i_{0}}_{\alpha_{0}} \times \bar{e}^{I}_{\alpha_{0}} ) \otimes_{C}  (e^{I}_{\alpha_{0}} \times \bar{e}^{i_{1}}_{\alpha_{1}} ) 
\otimes_{C} 
 (e^{i_{1}}_{\alpha_{1}} \times \bar{e}^{i_{2}}_{\alpha_{2}} )  
 \otimes_{C}  ... \otimes_{C}
  (e^{i_{k}}_{\alpha_{k}} \times \bar{e}^{i_{k+2}}_{\alpha_{k+2}} )  
\otimes_{C}  ... \otimes_{C}
(e^{i_{p}}_{\alpha_{p}} \times \bar{e}^{i_{0}}_{\alpha_{0}} )  
+  ....  \\
+ (-1)^{p-1}
(e^{i_{0}}_{\alpha_{0}} \times \bar{e}^{I}_{\alpha_{0}} ) \otimes_{C}  (e^{I}_{\alpha_{0}} \times \bar{e}^{i_{1}}_{\alpha_{1}} ) 
\otimes_{C} 
 (e^{i_{1}}_{\alpha_{1}} \times \bar{e}^{i_{2}}_{\alpha_{2}} )  
 \otimes_{C}  ... \otimes_{C}
(e^{i_{p-1}}_{\alpha_{p-1}} \times \bar{e}^{i_{0}}_{\alpha_{0}} ) ] = \\
= (e^{i_{0}}_{\alpha_{0}} \times \bar{e}^{I}_{\alpha_{0}} )     
  \otimes_{C}
 b' [  (e^{I}_{\alpha_{0}} \times \bar{e}^{i_{1}}_{\alpha_{1}} )
 \otimes_{C} 
 (e_{\alpha_{1}}^{i_{1}} \times \bar{e}^{i_{2}}_{\alpha_{2}} )  \otimes_{C}  ... \otimes_{C}  (e_{\alpha_{p}}^{i_{p}} \times \bar{e}^{i_{0}}_{\alpha_{0}} )] .
\end{multline}    
Summing up formulas (31) and (32) proves the proposition.
\end{proof}   
\begin{corollary}   
The complex     $\{  C^{\Delta}_{p}(\mathit{HS} (X)), b' \} $  is acyclic.
\end{corollary}  
\subsubsection{The Homology of the Sub-complex $ \{  C^{\Delta}_{p}(\mathit{HS} (X)), b \} $. }  
\begin{proposition}      
The operator $S$ establishes on the complex $K \{ C^{\Delta}_{p}(\mathit{HS} (X)), b \}$
a homotopy between the identity and the operator $\theta$
\begin{equation}
  b\; S_{(p)} + S_{(p-1)} \; b \;  = \; Id. - \theta,    
\end{equation}
where
\begin{multline}
\theta  [ (e^{i_{0}}_{\alpha_{0}} \times \bar{e}^{i_{1}}_{\alpha_{1}} )     
 \otimes_{C} 
 (e^{i_{1}}_{\alpha_{1}} \times \bar{e}^{i_{2}}_{\alpha_{2}} )  \otimes_{C}  ... \otimes_{C}  (e^{i_{p}} _{\alpha_{p}}\times \bar{e}^{i_{0}}_{\alpha_{0}} )]   = \\
 = (-1)^{p}  
 [ \;   (e^{i_{p}} _{\alpha_{p}} \times \bar{e}^{I}_{\alpha_{0}} )    
 \otimes_{C} 
 (e^{I}_{\alpha_{0}} \times \bar{e}^{i_{1}}_{\alpha_{1}} )    -
 (e^{i_{p}} _{\alpha_{p}}\times \bar{e}^{I}_{\alpha_{p}} )  
 \otimes_{C} 
(e^{I} _{\alpha_{p}}\times \bar{e}^{i_{1}}_{\alpha_{1}}) \;  ] 
 \otimes_{C}
 [ (e^{i_{1}}_{\alpha_{1}} \times \bar{e}^{i_{2}}_{\alpha_{2}} )  \otimes_{C}  ... \otimes_{C}   
 (e^{i_{p-1}} _{\alpha_{p-1}}\times \bar{e}^{i_{p}}_{\alpha_{p}} ) ] .
\end{multline}
\end{proposition}      
\begin{proof} For the proof of this proposition we use Proposition 11. One has
\begin{multline}     
( \; bS_{(p)} + S_{(p-1)} b \;) (K) =  \\
 = [\;b' + (-1)^{p+1} \partial_{(p+1)p+1}\;]S_{(p)} + S_{(p-1)} [\;b' + (-1)^{p} \partial_{(p)p} \;)](K) 
= \\
  = K +  (-1)^{p+1} \partial_{(p+1)p+1}S_{(p)} \;(K) + (-1)^{p}  S_{(p-1)}   \partial_{(p)p} \; (K)  = \\
= K + 
(-1)^{p+1}  \partial_{(p+1)p+1}
[ (e^{i_{0}}_{\alpha_{0}} \times \bar{e}^{I}_{\alpha_{0}} )    
 \otimes_{C} 
 (e^{I}_{\alpha_{0}} \times \bar{e}^{i_{1}}_{\alpha_{1}} )     
 \otimes_{C} 
 (e^{i_{1}}_{\alpha_{1}} \times \bar{e}^{i_{2}}_{\alpha_{2}} )  \otimes_{C}  ... \otimes_{C}  
 (e^{i_{p}} _{\alpha_{p}}\times \bar{e}^{i_{0}}_{\alpha_{0}} )]  \; + \\
 + (-1)^{p} S_{(p-1)}  
[ (e^{i_{p}} _{\alpha_{p}}\times \bar{e}^{i_{1}}_{\alpha_{1}} )  
 \otimes_{C} 
 (e^{i_{1}}_{\alpha_{1}} \times \bar{e}^{i_{2}}_{\alpha_{2}} )  \otimes_{C}  ... \otimes_{C}  
 (e^{i_{p-1}} _{\alpha_{p-1}}\times \bar{e}^{i_{p}}_{\alpha_{p}} )]  = \\
= K \;+ \;  (-1)^{p+1}     
  (e^{i_{p}} _{\alpha_{p}} \times \bar{e}^{I}_{\alpha_{0}} )    
 \otimes_{C} 
 (e^{I}_{\alpha_{0}} \times \bar{e}^{i_{1}}_{\alpha_{1}} )     
 \otimes_{C} 
 (e^{i_{1}}_{\alpha_{1}} \times \bar{e}^{i_{2}}_{\alpha_{2}} )  \otimes_{C}  ...  
  \otimes_{C} 
 (e^{i_{p-1}} _{\alpha_{p-1}}\times \bar{e}^{i_{p}}_{\alpha_{p}} )
 \otimes_{C}  
(e^{i_{p}} _{\alpha_{p}}\times \bar{e}^{i_{0}}_{\alpha_{0}} )  \; + \\
+ (-1)^{p} 
[ (e^{i_{p}} _{\alpha_{p}}\times \bar{e}^{I}_{\alpha_{p}} )  
 \otimes_{C} 
 (e^{I} _{\alpha_{p}}\times \bar{e}^{i_{1}}_{\alpha_{1}} )  
\otimes_{C} 
 (e^{i_{1}}_{\alpha_{1}} \times \bar{e}^{i_{2}}_{\alpha_{2}} )]  \otimes_{C}  ... \otimes_{C}   
 (e^{i_{p-1}} _{\alpha_{p-1}}\times \bar{e}^{i_{p}}_{\alpha_{p}} )]  =  \\
 = K \;-\; (-1)^{p}  
 [ \;   (e^{i_{p}} _{\alpha_{p}} \times \bar{e}^{I}_{\alpha_{0}} )    
 \otimes_{C} 
 (e^{I}_{\alpha_{0}} \times \bar{e}^{i_{1}}_{\alpha_{1}} )    -
 (e^{i_{p}} _{\alpha_{p}}\times \bar{e}^{I}_{\alpha_{p}} )  
 \otimes_{C} 
(e^{I} _{\alpha_{p}}\times \bar{e}^{i_{1}}_{\alpha_{1}}) \;  ] 
 \otimes_{C}
 [ (e^{i_{1}}_{\alpha_{1}} \times \bar{e}^{i_{2}}_{\alpha_{2}} )  \otimes_{C}  ... \otimes_{C}   
 (e^{i_{p-1}} _{\alpha_{p-1}}\times \bar{e}^{i_{p}}_{\alpha_{p}} ) ] ,
 \end{multline}
 which completes the proof of the Proposition 13.
\end{proof}  
\par
We intend to describe  the Hochschild boundary acting  on $C^{\Delta}_{p}(\mathit{HS} (X)) $  in a  
more topological fashion; this will help us to understand better the hidden geometry and suggest the next
constructions. 
For doing this we introduce the following 
 \begin{notation}        
To the elementary chain  $K \in C^{\Delta}_{p}(\mathit{HS} (X)) $  
 \begin{equation}     
K =  (e_{\alpha_{0}}^{i_{0}} \times \bar{e}^{i_{1}}_{\alpha_{1}} )
 \otimes_{C} 
 (e_{\alpha_{1}}^{i_{1}} \times \bar{e}^{i_{2}}_{\alpha_{2}} )  \otimes_{C}  ... \otimes_{C}  (e_{\alpha_{p}}^{i_{p}} \times \bar{e}^{i_{0}}_{\alpha_{0}} )  
 \end{equation}   
 we bi-univocally associate the symbol
 \begin{equation}     
 [K] =
 \begin{bmatrix}
    i_{1} & ...  &  i_{p-1}  &   i_{p}  &   i_{0}  \\
   \alpha_{1} &  ... &  \alpha_{p-1}  &  \alpha_{p}  &  \alpha_{0}
\end{bmatrix}   
 \end{equation}   
 \end{notation}   
 \begin{proposition}  
 Using the Notation 14, the Hochschild boundary faces are
 \begin{gather}    
 [b_{(p)k}K] =   (-1)^{k-1}
 \begin{bmatrix}
    i_{1}         & ...  & \hat{  i_{k} }       & ... &  i_{p-1}      &   i_{p}   &   i_{0}     \\
  \alpha_{1} &  ... &\hat{\alpha_{k}} & ...  & \alpha_{p-1}  & \alpha_{p}  & \alpha_{0}
\end{bmatrix}
 \text{, for $0 \leq k \leq p-1$} 
 \end{gather}   
 and
 \begin{equation}   
 [b_{(p)p}K] =  
 (-1)^{p}
 \begin{bmatrix}
     i_{1}         & ...  &  i_{k}       & ... &  i_{p-1}             &   i_{p}   &   \hat{i_{0 } }    \\
 \alpha_{1} &  ... &\alpha_{k} & ...  &  \alpha_{p-1}  &  \alpha_{p}  &  \hat {\alpha_{0} } 
\end{bmatrix}
 \end{equation}  
 \end{proposition}    
 \begin{definition}     
Let  $[\mathit{X}]^{\Delta}$  be the simplicial set whose vertices are the 
column-pairs $[i_{k} \alpha_{k}]$ and whose boundary faces $\partial_{(p)k}$ are the simplicial boundary faces. 
 \end{definition}   
 Proposition 15 implies the next result.
 \begin{theorem}   
 The linear bijective mapping $ [\;]: K \rightarrow [K]$ interchanges  
 \par
 i) the Hochschild boundary faces with the simplicial boundary faces and hence
 \par  
 ii)  $ [\;] $ establishes an isomorphism from the complex $C^{\Delta}_{p}(\mathit{HS} (X))$ to the simplicial chain complex
 of the simplicial space $[X]^{\Delta}$
 \begin{equation}   
 [b K] = \partial [K],  \hspace{0.5cm} for\;any\; K \in C^{\Delta}_{p}(\mathit{HS} (X)).
 \end{equation}   
 \end{theorem}  
 \begin{theorem}  
i) The Hochschild homology of the algebra of Hilbert-Schmidt operators on $X$ is isomorphic to the simplicial homology of $[X]^{\Delta}$, (see Definition 16).
\par
ii) The simplicial complex $[X]^{\Delta}$ is an infinite dimensional simplex.
\end{theorem}  

\begin{proof}  
i) is a consequence of Theorem 17 -ii).
\par
ii) In the simplicial complex $[X]^{\Delta}$ any column-vertex $[i_{k}, \alpha_{k}]$  may be joint with any  column-vertex  $[i_{k+1}, \alpha_{k+1}]$ to create the simplex  (37).  In other words, the simplicial complex $[X]^{\Delta}$ is an infinite dimensional simplex and therefore its homology is trivial.
\end{proof}  

 \begin{theorem}   
 Let $X$ be a countable, locally finite simplicial complex of dimension $n$. Then, the Hochschild homology of the algebra of real or complex Hilbert-Schmidt operators on $X$ is trivial. 
\end{theorem}  
It is known that the Hochschild homology of Banach algebras "is not interesting", see Connes \cite{Connes}, or that in many cases it is trivial, see e.g. Gr¿nb¾k  \cite{ Gr¿nb¾k }, or  that "cyclic homology is degenerate on Banach or $C^{\ast}$-algebras", see \cite{Cuntz}, p. 42 and Proposition 3.5, Corollary 3.6., p. 39.  
The basic intent of this paper is to contribute to the better understanding of the interplay between the analysis and the topology at the level of Hochschild homology.  More specifically, we state that the relationship between the content of Theorem 19 and the Theorem 1, or the relationship between the Hochschild homology with no control on the chain-supports and {\em local} Hochschild homology \cite{Teleman_arXiv}, is the same relationship which occurs in the definition of the Alexander-Spanier (co)-homology before and after the consideration of the control on the supports, see \S 8.1.

\begin{proof}  
Given that the homology of the simplex is trivial, Theorem 19 is an immediate consequence of the Theorem 18.  
\end{proof}  
\par
The Notation 14 allows us to describe the operator $\theta$ in a more geometrical fashion, allowing simplices to be multiplied by integers
\begin{multline}
\theta                    
\begin{bmatrix}
  {  i_{1} }& ...  &{  i_{p-1} } & {  i_{p} }   &  i_{0}    \\
   { \alpha_{1} }&  ... & { \alpha_{p-1} } & { \alpha_{p} }  &   { \alpha_{0} }
\end{bmatrix}
= \\
 = (-1)^{p+1} [
\begin{bmatrix}
              I          &     i_{1}        &   i_{2}        & ... &     i_{p-1}    &  i_{p}       \\
    { \alpha_{p} }&  \alpha_{1} & \alpha_{2}  & ... &  \alpha_{p-1}   &  \alpha_{p}
\end{bmatrix}
- 
\begin{bmatrix}
       I          &    i_{1}       &   i_{2}         & ... &{  i_{p-1} }    &    i_{p}        \\
   \alpha_{0}   &  \alpha_{1} & \alpha_{2}  & ... & { \alpha_{p-1} }  & \alpha_{p} 
\end{bmatrix}
],
\end{multline}   
which may be written further

\begin{equation*}        
\theta
\begin{bmatrix}
  {  i_{1} }& ...  &{  i_{p-1} } & {  i_{p} }   & {  i_{0} }   \\
  { \alpha_{1} }&  ... & { \alpha_{p-1} } & { \alpha_{p} }  &   { \alpha_{0} } 
\end{bmatrix}
= 
\end{equation*}
\begin{gather}
=  (-1)^{p+1} 
( 
\begin{bmatrix}
             I        \\
    \alpha_{p}  
\end{bmatrix}
  - 
\begin{bmatrix}
       I       \\
   \alpha_{0}    
\end{bmatrix}
)
\begin{bmatrix}
    i_{1}       &   i_{2}         & ... &{  i_{p-1} }     & i_{p}     \\
     \alpha_{1} & \alpha_{2}  & ... & { \alpha_{p-1} }  & \alpha_{p}
\end{bmatrix}
\end{gather}          
\par
Formula (42) helps us to iterate  the operator $\theta$ easily.
\begin{remark}     
1. The operator $\theta$ removes the last vertex $[i_{0}, \alpha_{0}]$, replaces it with the difference  
$[I, \alpha_{p}] - [I, \alpha_{0}] $ and places it in the first position
\par
2. The operator $\theta$ does not affect the other vertices, except for a cyclical shift of them to the right.
\end{remark}    

\begin{proposition}    
The operator $\theta$ satisfies
\par
-i)
\begin{equation*}        
\theta^{p}
\begin{bmatrix}
  {  i_{1} }& ...  &{  i_{p-1} } & {  i_{p} }   & {  i_{0} }   \\
  { \alpha_{1} }&  ... & { \alpha_{p-1} } & { \alpha_{p} }  &  { \alpha_{0} } 
\end{bmatrix}
= 
\end{equation*}
\begin{gather}
=                 
( 
\begin{bmatrix}
             I        \\
    \alpha_{1}  
\end{bmatrix}
  - 
\begin{bmatrix}
       I       \\
   \alpha_{2}    
\end{bmatrix}
)
( 
\begin{bmatrix}
             I        \\
    \alpha_{2}  
\end{bmatrix}
  - 
\begin{bmatrix}
       I       \\
   \alpha_{3}    
\end{bmatrix}
)
.....
( 
\begin{bmatrix}
             I        \\
    \alpha_{p-1}  
\end{bmatrix}
  - 
\begin{bmatrix}
       I       \\
   \alpha_{p}    
\end{bmatrix}
)
( 
\begin{bmatrix}
             I        \\
    \alpha_{p}  
\end{bmatrix}
  - 
\begin{bmatrix}
       I       \\
   \alpha_{0}    
\end{bmatrix}
)
\begin{bmatrix}
             i_{1}       \\
    \alpha_{1}  
\end{bmatrix}
\end{gather}        
\par
-ii)

\begin{equation*}        
(-1)^{p+1} \; \theta^{p+1}
\begin{bmatrix}
  {  i_{1} }& ...  &{  i_{p-1} } & {  i_{p} }   &  {  i_{0} }    \\
  { \alpha_{1} }&  ... & { \alpha_{p-1} } & { \alpha_{p} }  &  { \alpha_{0} }
\end{bmatrix}
= 
\end{equation*}
\begin{gather}   
=  
( 
\begin{bmatrix}
             I        \\
    \alpha_{p}  
\end{bmatrix}
  - 
\begin{bmatrix}
       I       \\
   \alpha_{1}    
\end{bmatrix}
)
( 
\begin{bmatrix}
             I        \\
    \alpha_{2}  
\end{bmatrix}
  - 
\begin{bmatrix}
       I       \\
   \alpha_{1}    
\end{bmatrix}
)
.....
( 
\begin{bmatrix}
             I        \\
    \alpha_{p-1}  
\end{bmatrix}
  - 
\begin{bmatrix}
       I       \\
   \alpha_{p-2}    
\end{bmatrix}
)
( 
\begin{bmatrix}
             I        \\
    \alpha_{p}  
\end{bmatrix}
  - 
\begin{bmatrix}
       I       \\
   \alpha_{p-1}    
\end{bmatrix}
) 
\begin{bmatrix}
             I       \\
    \alpha_{p}  
\end{bmatrix}
\end{gather}         

\begin{gather}   
-
( 
\begin{bmatrix}
             I        \\
    \alpha_{0}  
\end{bmatrix}
  - 
\begin{bmatrix}
       I       \\
   \alpha_{1}    
\end{bmatrix}
)
( 
\begin{bmatrix}
             I        \\
    \alpha_{2}  
\end{bmatrix}
  - 
\begin{bmatrix}
       I       \\
   \alpha_{1}    
\end{bmatrix}
)
.....
( 
\begin{bmatrix}
             I        \\
    \alpha_{p-1}  
\end{bmatrix}
  - 
\begin{bmatrix}
       I       \\
   \alpha_{p-2}    
\end{bmatrix}
)
( 
\begin{bmatrix}
             I        \\
    \alpha_{p}  
\end{bmatrix}
  - 
\begin{bmatrix}
       I       \\
   \alpha_{p-1}    
\end{bmatrix}
) .
 \begin{bmatrix}
             I       \\
    \alpha_{0}  
\end{bmatrix}
\end{gather}      
\begin{gather}   
=
( 
\begin{bmatrix}
             I        \\
    \alpha_{p}  
\end{bmatrix}
  - 
\begin{bmatrix}
       I       \\
   \alpha_{1}    
\end{bmatrix}
)
( 
\begin{bmatrix}
             I        \\
    \alpha_{2}  
\end{bmatrix}
  - 
\begin{bmatrix}
       I       \\
   \alpha_{1}    
\end{bmatrix}
)
.....
( 
\begin{bmatrix}
             I        \\
    \alpha_{p-1}  
\end{bmatrix}
  - 
\begin{bmatrix}
       I       \\
   \alpha_{p-2}    
\end{bmatrix}
)
( 
\begin{bmatrix}
             I        \\
    \alpha_{p}  
\end{bmatrix}
  - 
\begin{bmatrix}
       I       \\
   \alpha_{p-1}    
\end{bmatrix}
) 
(
\begin{bmatrix}
             I       \\
    \alpha_{p}  
\end{bmatrix}
-
 \begin{bmatrix}
             I       \\
    \alpha_{0}  
\end{bmatrix}
)
\end{gather}   

\end{proposition}    

\subsection{Homological Consequences}   
\par

\begin{definition}           
-i) Let $C^{I}_{p}(HS) \subset C^{\Delta}_{p}(HS)$ denote the subset consisting of all $p$-chains whose elementary components 
contain exclusively  the Hilbert space basis index $I$.
\par
 $\{C^{I}_{p}(HS), \;b\}_{0 \leq p}$ is complex. This complex will be called {\em reduced diagonal complex.}
\par
-ii) Let $[X]^{I}$ be the sub-simplex of the simplex $[X]^{\Delta}$ whose vertices are the the column pairs $[I, \alpha]$. 
\end{definition}

\begin{lemma}       
-The  spaces $C^{I}_{p}(HS) $  have the properties
\par
i) for any $0 \leq p + 1 \leq \tilde{p}$, 
\begin{equation}  
\theta_{(p)}^{\tilde{p}} ( C^{\Delta}_{p}(HS)) \subset  C^{I}_{p}(HS)
\end{equation}   
\par
ii) $b,\; S, \; \theta^{k}$  transform the spaces $C^{I}_{\ast}(HS) $  into themselves, for any $  0 \leq k $, 
\par
iii) $\{  \;   C^{I}_{p}(HS) , b  \; \}_{0\leq p}$  is a sub-complex of  $\{  \;   C^{\Delta}_{p}(HS) , b  \; \}_{0\leq p}$ 
\end{lemma}      

\begin{proof}   
 Relations (42), (43) along with (40) imply -i).
 \par
 -ii) The equations (37)-(39) imply the property relative to $b$.  The Definition 10, equation (29) implies the property concerning  $S$  and the equation (41)-(43) prove the property concerning   $\theta^{k}$.
\end{proof}    

\par
\begin{definition}    
 For any $0 \leq \tilde{p} $, we define  
\begin{equation}       
\tilde{S}_{(p)}:    C^{\Delta}_{p}(HS)   \longrightarrow C^{\Delta}_{p+1}(HS)    
\end{equation}            
be given by
\begin{equation}   
\tilde{S}_{(p)} :=    ( 1 + \theta_{(p)}  +  \theta_{(p)}^{2} + ... +  \theta_{(p)}^{p+1} ) S_{(p)}.   
\end{equation}     
We agree to represent the homomorphisms  $\tilde{S}_{(p)}$ with the same notation when they are restricted to the subspaces
$ C^{I}_{p}(HS) $.
\end{definition}      
\begin{remark}  
In the formula (49) the extra power $\theta_{(p)}^{p+1}$ is inserted. This extra power will be used in the proof of 
Proposition 27. -ii) below, formula (53).         
\end{remark}    
\begin{proposition}
\par
-i) The operators $\theta_{(p)}$ commutes with $b$
\begin{equation}       
 \theta_{(p-1)} \; b \; = \; b \; \theta_{(p)}         
\end{equation}         
\par
-ii) The operators $\tilde{S}_{(p)}$ satisfy
\begin{equation}      
 \tilde{S}_{(p-1)}b_{(p)} + b_{(p+1)} \tilde{S}_{(p)}  = 1 - \theta_{(p)}^{p+2}    
\end {equation}         
both on the spaces  $C^{\Delta}_{p}(HS)$  and $C^{I}_{p}(HS)$, for any $0 \leq p$.
\end{proposition}   
 
\begin{proof}    
-i)  The commutation relation is obtained by multiplying the relation (33) to the left and to the right by $b$
\begin{equation}     
 bSb = b(bS + Sb)  = b - b \theta,  \hspace{1cm}   bSb = (bS + Sb)b  = b -  \theta b.    
\end{equation}        
\par
-ii) The identity (51) is obtained by multiplying the equation (33) to the left by      
$ ( 1 + \theta_{(p)}  +  \theta_{(p)}^{2} + ... +  \theta_{(p)}^{p+1} )$, along with the commutation identity (50).   
\end{proof} 
\begin{proposition}   
\par
-i) Any  $p$-cycle $ K \in  C^{\Delta}_{p}(HS) $ is co-homologous, inside the complex $\{  \;   C^{\Delta}_{p}(HS) , b  \; \}_{0\leq p}$  
to the cycle $\theta_{(p)}^{\tilde{p}} K  \in C^{I}_{p}(HS)$, for any $p+1 \leq \tilde{p}$
\par
-ii) If $K \in C^{I}_{p}(HS) $ is a boundary inside the complex $\{  \;   C^{\Delta}_{p}(HS) , b  \; \}_{0\leq p}$  
( $K_{(p)} = bL_{(p+1)}$,  where $L_{(p+1)} \in C^{\Delta}_{p+1}(HS)$),  then $K$ is a boundary inside the complex
$\{  \;   C^{I}_{p}(HS) , b  \; \}_{0\leq p} $  ($K_{(p)} = b \tilde{L}_{(p+1)}$, 
where  $\tilde{L}_{(p+1)}  \; =  \;  \tilde{S}_{(p)} K_{(p)}  +  \theta_{(p+1)}^{p+2} \; L_{(p+1)}  
 \in  C^{I}_{p}(HS) $). 
\end{proposition}    

\begin{proof}    
-i)   For any cycle $K \in C^{\Delta}_{p}(HS) $, the relation  (51) gives     
\begin{equation}   
( \tilde{S}_{(p-1)}b_{(p)} + b_{(p+1)} \tilde{S}_{(p)})  ( K ) = (1 - \theta_{(p)}^{p+2}) (K). 
\end{equation}      
and then
\begin{equation}      
K = \theta_{(p)}^{p+2} K  +  b_{(p+1)} ( \tilde{S}_{(p)}  K ).
\end{equation}      
Equation  (47) insures that $  \theta_{(p)}^{p+2} K \in C^{I}_{p}(HS) $ while the commutation relation (50) implies that  $\theta_{(p)}^{p+2} K$ is a cycle
 \begin{equation}    
b \; (\; \theta_{(p)}^{p+2} K\; ) = \theta_{(p)}^{p+2} ( b K ) = 0.
 \end{equation}    
 Equation (54) complemented with this information proves  part i). 
 \par
 Notice that here we do not make any statement about whether or not $ \tilde{S}_{(p)}  K $  does belong to the sub-complex  
 $C^{I}_{p+1}(HS) $.  Part ii) of Proposition 27 clarifies this point.
 
\par
-ii)  By hypothesis, $K \in C^{I}_{p}(HS) $ is a boundary inside the complex $\{  \;   C^{\Delta}_{p}(HS) , b  \; \}_{0\leq p}$;   then $K$ is  a cycle in the complex $\{  \;   C^{I}_{p}(HS) , b  \; \}_{0\leq p}$.
\par
We plug $K =  bL$,  into the identity (51) to get
\begin{equation*}   
( \tilde{S}_{(p-1)}b_{(p)} + b_{(p+1)} \tilde{S}_{(p)})  ( bL ) = (1 - \theta_{(p)}^{p+2}) (bL) 
\end{equation*}  
or
\begin{equation*}  
 b_{(p+1)} \tilde{S}_{(p)}  ( bL ) = K - \theta_{p}^{p+2} (bL),
\end{equation*}   
which gives
\begin{gather}     
K_{(p)} =  b_{(p+1)} \tilde{S}_{(p)}  ( bL_{(p+1)} ) +  \theta_{(p)}^{p+2} (bL_{(p+1)}).
\end{gather}     
\par
The commutation relation (50) gives us further
\begin{gather*}   
K_{(p)}   
= b_{(p+1)}  \tilde{S}_{(p)}  ( bL_{(p+1)} ) +   b_{(p+1)}  \theta_{(p+1)}^{p+2} (L_{(p+1)})  = \\
 = b_{(p+1)}  ( \; \tilde{S}_{(p)}  ( bL_{(p+1)} ) +    \theta_{(p+1)}^{p+2} (L_{(p+1)}) \;  )
\end{gather*}    
which can be re-writen 
\begin{equation}    
K_{(p)} =   b_{(p+1)} \; ( \; \tilde{S}_{(p)}  K_{(p)}  +    \theta_{p+1}^{p+2} L_{(p+1)} \;  ).
\end{equation}          
By hypotheses $K_{(p)} \in C^{I}_{p}(HS) $. Then Lemma 23, ii) gives that 
\begin{equation}    
 \tilde{S}_{(p)}  ( K_{(p)} )  \in  C^{I}_{p+1}(HS).
 \end{equation}  
On the other hand,  Lemma 23 i), formula (47) insures that  
\begin{equation}     
\theta_{(p+1)}^{p+2} (L_{(p+1)}) \in  C^{I}_{p+1}(HS).
\end{equation}     
The equation (57) may be re-written
\begin{equation*}
K = b \; \tilde{L}_{(p+1)}
\end{equation*}
where, in virtue of  the relations (58),  (59) one has
\begin{equation}  
\tilde{L}_{(p+1)} = \; \tilde{S}_{(p)}  K_{(p)}  +    \theta_{p+1}^{p+2} L_{(p+1)}  \;  \in \; C^{I}_{p+1}(HS).  
\end{equation}  
This completes the proof of ii).
\end{proof}  
\begin{theorem}  
  The inclusion of complexes $ \iota: \{  \;   C^{I}_{p}(HS) , b  \; \}_{0\leq p}  \longrightarrow \{  \;   C^{\Delta}_{p}(HS) , b  \; \}_{0\leq p}$  
induces isomorphism in homology.
\par
Therefore, the homology of the reduced diagonal complex is the homology of the diagonal complex.
\end{theorem}  
\begin{proof}  
 Part i) of Proposition 27 tells that any homology class in the complex  $\{  \;   C^{\Delta}_{p}(HS) , b  \; \}_{0\leq p}$  has a cycle representative in the sub-complex $\{  \;   C^{I}_{p}(HS) , b  \; \}_{0\leq p}$. Therefore, the inclusion $\iota$
induces epimorphisms in homology. 
 \par
 The second part ii) of the same proposition tells us that if the homology class of the cycle $K$ of the sub-complex 
 $\{  \;   C^{I}_{p}(HS) , b  \; \}_{0\leq p}$
 is a boundary in the complex $\{  \;   C^{\Delta}_{p}(HS) , b  \; \}_{0\leq p}$, then it is a boundary in the sub-complex $\{  \;   C^{I}_{p}(HS) , b  \; \}_{0\leq p}$  too.  In other words, the inclusion  $\iota$ induces monomorphisms in homology.
 This completes the proof of the Theorem 28.
 \end{proof}              
\par
Proposition 8. ii) and Proposition 28 imply the first part of the following
\begin{theorem}     
i) For any locally finite, countable,  homogeneous, simplicial complex $X$, the continuous Hochschild homology of the algebra of Hilbert-Schmidt operators on $X$ is isomorphic to the homology of the reduced diagonal sub-complex 
$\{  \;   C^{I}_{p}(HS) , b  \; \}_{0\leq p}$. 
\par
ii) The mapping $[ \; ]$  induces an isomorphism from the sub-complex  $\{  \;   C^{I}_{p}(HS) , b  \; \}_{0\leq p}$ 
to the space of simplicial chains of the complex $[X]^{I}$,
\par
-iii) the homology of the sub-complex  $\{  \;   C^{I}_{p}(HS) , b  \; \}_{0\leq p}$  is trivial.
\end{theorem}  
\begin{proof}        
Looking back at the Theorem 17, it is clear that the mapping $[ \; ]$ establishes also an isomorphism from the complex $\{C^{I}_{\ast}(HS), b\} $ to the chain complex, with complex coefficients, of the space $[X]^{I}$ (see Definition 22 -ii).  As the simplicial complex $[X]^{I}$ is a simplex, the result follows.
\end{proof}          
\begin{remark}   
i) Theorem 29 implies the triviality of the continuous Hochschild homology of the algebra of Hilbert-Schmidt operators. This result is not new, see reference after Theorem 19. In \S 8. we shall discuss how this result changes when we will be considering continuous Hochschild chains with small supports about the main diagonal, i.e. when we will be going to compute the {\em local} Hochschild homology of the algebra of Hilbert-Schmidt operators. 
\par
To complete the proof of Theorem 1, we will use the sub-complex $\{  \;   C^{I, loc}_{p}(HS) , b  \; \}_{0\leq p}$ of 
$\{  \;   C^{I}_{p}(HS) , b  \; \}_{0\leq p}$ consisting of those chains which have small supports. To be able to do this, we shall observe that the quasi-isomorphisms treated in \S 5 and \S 6
\begin{equation*}
\{  \;   C_{\ast}(HS) , b  \; \} \longleftarrow \{  \;  C^{\Delta}_{\ast}(HS) , b  \; \}  \longleftarrow  \{ \;  C^{I}_{\ast}(HS) , b \; \}
\end{equation*}
 pass to the {\em local} sub-complex. 
 \par
 ii) It is interesting to investigate more closely the homotopy operator  $\tilde{S}$.
\end{remark}   


\section{Analytic Considerations}   
\par
In this section we regard the chains of the complex $C_{\ast}(HS(X))$ and we discuss both the continuity of the Hochschild boundary $b$ and of the homotopy operators $s_{(p)}$,   $S_{(p)}$ on this complex.

\subsection{Continuous Hochschild Chains over the Algebra of Hilbert-Schmidt Operators.}  
\par
Let $\{ e_{\alpha}^{n}  \}_{n \in N}$ be an ortho-normal basis of $L_{2}$ complex valued functions on $\Delta_{\alpha}$.
Then $\{ e^{i}_{\alpha}  \times  \bar{e}^{j}_{\beta}  \}_{i, j \in N} $ is an ortho-normal basis of $L_{2}$ complex valued functions on $\Delta_{\alpha} \times \Delta_{\beta}$.

\par
Any Hilbert-Schmidt kernel on $X \times X$ is given by an $L_{2}$-convergent series
\begin{equation}       
 K \;=\; \sum_{\alpha \beta, ij} K_{ij}^{\alpha \beta} \; (e_{\alpha}^{i} \times \bar{e}^{j}_{\beta} ),  \hspace{1cm}
 \text{$\sum_{i,j} | K_{i,j}^{\alpha \beta} |^{2} < \infty$}
\end{equation}     
with complex coefficients $K_{ij}^{\alpha \beta}$. 
\par
Given the Hilbert-Schmidt operator $K$, the decomposition (1) is unique.
A Hilbert-Schmidt operator of type $(e_{\alpha}^{i} \times \bar{e}^{j}_{\beta}) $ was called {\em elementary}.
\par
The composition of two elementary Hilbert-Schmidt operators is given by
\begin{equation}                  
(e_{\alpha}^{i} \times \bar{e}^{j}_{\beta}) \circ (e_{\gamma}^{k} \times \bar{e}^{l}_{\eta} ) \;=\;  
             \delta^{jk}  \delta_{\beta \gamma} \; (e_{\alpha}^{i} \times \bar{e}^{l}_{\eta}),
\end{equation}  
where $\delta^{jk}$ and $\delta_{\beta \gamma}$ are the Kronecker symbols.

\par
Given two Hilbert-Schmidt kernels
\begin{equation}       
K \;=\; \sum_{\alpha \beta, ij} K_{ij}^{\alpha \beta} \; (e_{\alpha}^{i} \times \bar{e}^{j}_{\beta} ),  \hspace{0.5cm}
L \;=\; \sum_{\alpha \beta, ij} L_{ij}^{\alpha \beta} \; (e_{\alpha}^{i} \times \bar{e}^{j}_{\beta} ),
\end{equation}   
their composition is given by the formula

\begin{equation}   
K \circ L = 
 \sum_{\alpha \beta, ij} \; K_{ij}^{\alpha \beta} \;  L_{kl}^{\alpha \beta}  \; \delta^{j,k} \;
 (e_{\alpha}^{i} \times \bar{e}^{l}_{\beta} ) 
= 
 \sum_{\alpha, \beta, \gamma} ^{i,k,j} \; K_{ik}^{\alpha \beta} \;  L_{kl}^{ \beta \gamma}  \;
 (e_{\alpha}^{i} \times \bar{e}^{l}_{\beta} ),
\end{equation}    

i.e.
\begin{equation}   
(K \circ L)^{\alpha, \gamma}_{i, l}  = \sum_{k}^{\beta}   \; K_{ik}^{\alpha \beta} \;  L_{kl}^{ \beta \gamma} 
\end{equation}  

\begin{definition}    
The space of continuous $p$-chains of the algebra of Hilbert-Schmidt operators on $X$ is
\begin{equation}   
C_{p} (\mathit{HS})(X) := 
\end{equation}
\begin{gather*}
= \{  \sum_{i, j}^{\alpha, \beta}  K_{i_{0},...,i_{p},j_{0},..,j_{p} }^{\alpha_{0},..,\alpha_{p}, \beta_{0},...,\beta_{p}} 
(e^{i_{0}}_{\alpha_{0}}  \times \bar{e}^{j_{0}}_{\beta_{0}} )  
\otimes_{\mathit{C}}  (e^{i_{1}}_{\alpha_{1}}  \times \bar{e}^{j_{1}}_{\beta_{1}} ) 
\otimes_{\mathit{C}}  .... 
\otimes_{\mathit{C}}  (e^{i_{p}}_{\alpha_{p}}  \times \bar{e}^{j_{p}}_{\beta_{p}} ) \; | \;  \\
\text{ such that  $ \sum_{i,j}^{\alpha, \beta} |K_{i_{0},..,i_{p}, j_{0},...,j_{p}}^{\alpha_{0},..,\alpha_{p},j_{0},...,j_{p}} |^{2} < \infty $} 
  \}
\end{gather*}     
The coefficients $ K_{i_{0},...,i_{p},j_{0},..,j_{p} }^{\alpha_{0},..,\alpha_{p}, \beta_{0},...,\beta_{p}}$ 
 represent the components of the $L_{2}$-decomposition of the chain $K$  relative to the Hilbert basis
\begin{equation}    
\{  (e^{i_{0}}_{\alpha_{0}}  \times \bar{e}^{j_{0}}_{\beta_{0}} )  
\otimes_{\mathit{C}}  (e^{i_{1}}_{\alpha_{1}}  \times \bar{e}^{j_{1}}_{\beta_{1}} ) 
\otimes_{\mathit{C}}  .... 
\otimes_{\mathit{C}}  (e^{i_{p}}_{\alpha_{p}}  \times \bar{e}^{j_{p}}_{\beta_{p}} ) \}_{\alpha, \beta, i, j}.
\end{equation}   
By definition, the norm of $K$ is given by
\begin{equation}     
||K||^{2} := \; \sum_{i,j}^{\alpha, \beta} |K_{i_{0},..,i_{p}, j_{0},...,j_{p}}^{\alpha_{0},..,\alpha_{p},j_{0},...,j_{p}} |^{2}  < \infty.
\end{equation}   
\end{definition}       

We discuss here the continuity property both of the Hochschild boundary and of the homotopy operators $s$, $S$ defined on the spaces of continuous chains $C_{p}(HS(X))$.
\subsection{Continuity of the Hochschild boundary.}        
\par
We begin by addressing the continuity property of the Hochschild boundary. 
\begin{proposition}       
-i) The series 
\begin{equation}   
(K \circ L)_{i, l}  := \sum_{k}  \; K_{ik}^{\alpha \beta} \;  L_{kl}^{ \beta \gamma} 
\end{equation}     
is absolutely summable.
\par
-ii) The coefficients $(K \circ L)_{i, l} $ satisfy
\begin{equation}  
 \sum_{i,j} \; |(K \circ L)_{i, l}|^{2}  \leq   (\sum_{i,j} \; |K_{i, l}|^{2}  )  .  ( \sum_{i,j} | L_{i, l}|^{2} )
\end{equation}   
\end{proposition}     

\begin{proof}  
-i) The Cauchy-Schwartz inequality gives
\begin{gather}
 \sum_{k}  \; | K_{ik} \;  L_{kl} | \; \leq \; ||   ( K_{ik} )_{k}   || . ||   (  L_{kl} )_{k}    || \; =  
%
%
\;   ( \sum_{k}  |K_{ik}|^{2} )^{1/2}    .  ( \sum_{k}   | L_{kl} |^{2} )^{1/2}  
\end{gather}
\par
From -i) we get
\begin{equation}   
 \sum_{i,l}  ( \sum_{k}  \; | K_{ik} \;  L_{kl} | )^{2} \; \leq \;  \sum_{i,j} ||   ( K_{ik} )_{k}   ||^{2} . ||   (  L_{kl} )_{k}    ||^{2} \; =  
%
%
\;   \sum_{i,l} ( \sum_{k}  |K_{ik}|^{2} )    .  ( \sum_{k}   | L_{kl} |^{2} )  \leq 
\end{equation}
\begin{equation*}    
\leq \;  \sum_{i,l} ( \sum_{k}  |K_{ik}|^{2} )    .  ( \sum_{k}   | L_{kl} |^{2} )  \; \leq \;  
( \sum_{i,j}   |K_{ij}|^{2} )    .  ( \sum_{k,l}   | L_{kl} |^{2} )  = || K  ||_{HS}^{2} .   || L  ||_{HS}^{2}
\end{equation*}    
\end{proof}   

\par
Recall that the Hochschild boundary $b_{(p)}: C_{p} (\mathit{HS} (X)) \longrightarrow C_{p-1} (\mathit{HS} (X))$ is 
\begin{equation*} 
b_{(p)} = \sum_{k=0}^{k=p-1}  b_{(p)k} +  b_{(p)p}     
\end{equation*}
where
\begin{equation*}
b_{(p)k} = (-1)^{k} \partial_{k}^{H} ,    \hspace{0.2cm}  and  \hspace{0.3cm} b_{(p)p} = (-1)^{p} \partial_{p}^{H}    
\end{equation*}
with
\begin{equation*}
\partial_{k}^{H}  (K_{0} \otimes_{C}  K_{1}  \otimes_{C} .... \otimes_{C} K_{p}  )  \;=\;           
        K_{0} \otimes_{C} ... \otimes_{C} K_{k-1} \otimes_{C} (K_{k} \circ  K_{k+1}) \otimes_{C} ... \otimes_{C} K_{p} 
\end{equation*}
and
\begin{equation*}     
\partial_{p} ^{H} (K_{0} \otimes_{C}  K_{1}  \otimes_{C} .... \otimes_{C} K_{p}  )  \;=\;     
         ( K_{p} \circ K_{0} ) \otimes_{C}  K_{1}  \otimes_{C} .... \otimes_{C} K_{p-1} .
\end{equation*}
\par
\begin{proposition}     
The Hochschild boundary face operators are well defined operators from the vector space $C_{p}(HS(X))$ of continuous Hochschild chains of the algebra of Hilbert-Schmidt operators into the space $C_{p-1}(HS(X))$ of continuous Hochschild chains of the algebra of Hilbert-Schmidt operators.
\end{proposition}   

\par

\subsection{Continuity of the Homotopy Operators $s$.}    
Here we show that the homotopy operators $s$ are well defined on the spaces of continuous Hochschild chains generated 
by elementary chains which contain at least one gap.
\par
Suppose $K$ is an elementary chain containing at least one gap. Suppose the first gap is of order $r$. Recall
the operator $s_{(p)} (K)$ is defined by inserting the factor $ (e^{j_{r}}_{\beta_{r}}  \times \bar{e}^{j_{r}}_{\beta_{r}} )$ inside this  gap. 
\par
\begin{equation*}     
s_{(p)} (K) :=     
\end{equation*}
\begin{equation}
= (-1)^{r}
(e^{i_{0}}_{\alpha_{0}}  \times \bar{e}^{j_{0}}_{\beta_{0}} )  \otimes_{\mathit{C}} ... 
\otimes_{\mathit{C}}  (e^{i_{r}}_{\alpha_{r}}  \times \bar{e}^{j_{r}}_{\beta_{r}} ) 
\otimes_{\mathit{C}}  (e^{j_{r}}_{\beta_{r}}  \times \bar{e}^{j_{r}}_{\beta_{r}} )
\otimes_{\mathit{C}}  (e^{i_{r+1}}_{\alpha_{r+1}}  \times \bar{e}^{j_{r+1}}_{\beta_{r+1}} ) \otimes_{C}... 
\otimes_{\mathit{C}}  (e^{i_{p}}_{\alpha_{p}}  \times \bar{e}^{j_{p}}_{\beta_{p}} ).
\end{equation}   
\par
 \begin{proposition}  
The homotopy operators $s$ satisfy
\begin{equation}   
|| s K ||  =  || K ||.
\end{equation}   
and hence are well defined on the Hochschild sub-complex of continuous chains over the algebra of Hilbert-Schmidt operators. 
\end{proposition}   
\begin{proof}  
-i) Let  $K$ the above chain and let  $K_{i_{0},...,i_{p},j_{0},..,j_{p} }^{\alpha_{0},..,\alpha_{p}, \beta_{0},...,\beta_{p}} $ be its
components.  Then the components of $sK$ are
\begin{gather}     
(sK)_{ \;i_{0},...,i_{r}, j_{r},i_{r+1}, ..., \; \;i_{p}, \;j_{0},...,\;j_{r},j_{r},\;j_{r+1},...,j_{p} }
^{ \alpha_{0},..,\alpha_{r}, \beta_{r}, \alpha_{r+1},...,\alpha_{p}, \beta_{0}, ...,\beta_{r}, \beta_{r}, \beta_{r+1}, ...,\beta_{p} } = 
 (-1)^{r} K_{ \;i_{0},...,i_{r},\;i_{r+1}, ..., \;i_{p}, \;j_{0},...,\;j_{r},j_{r+1},...,\;j_{p} }
^{ \alpha_{0},..,\alpha_{r}, \alpha_{r+1},...,\alpha_{p}, \beta_{0}, ...,\beta_{r}, \beta_{r+1}, ...,\beta_{p} } .
\end{gather}   
\par

We have to check that the components of $sK$ satisfy the condition (68).  This is clearly true. Indeed, the formula (75) tells us that the passage from the components of $K$ to the components of $sK$ involves two operations: a shift (given by the addition of two columns $(j_{r}, j_{r}, \beta_{r}, \beta_{r}$) and the multiplication by $(-1)^{r}$. All other components, which do not contain the additional columns, are equal to zero.  The additional columns are uniquely characterised by the choice of the {\em first} gap. The shift prevents non-trivial linear combinations between the components of $K$ to produce the components of $sK$.  These arguments prove 
\begin{equation*}   
|| s K ||  =  || K ||.
\end{equation*}   
which completes the proof of the proposition.
\end{proof}   

\subsection{ Continuity of the Homotopy Operators $S$}   
\begin{proposition}     
The operator 
$S_{(p)}:    C^{\Delta}_{p}(\mathit{HS} (X)) \longrightarrow C^{\Delta}_{p+1}(\mathit{HS} (X)) $,  defined on elementary chains by the formula
\begin{multline}
S_{(p)} 
[(e_{\alpha_{0}}^{i_{0}} \times \bar{e}^{i_{1}}_{\alpha_{1}} )
 \otimes_{C} 
 (e_{\alpha_{1}}^{i_{1}} \times \bar{e}^{i_{2}}_{\alpha_{2}} )  
 \otimes_{C}  ... \otimes_{C}  
 (e_{\alpha_{p}}^{i_{p}} \times \bar{e}^{i_{0}}_{\alpha_{0}} )] := \\
:=  (e_{\alpha_{0}}^{i_{0}} \times \bar{e}^{I}_{\alpha_{0}} )   
  \otimes_{C} 
  (e_{\alpha_{0}}^{I} \times \bar{e}^{i_{1}}_{\alpha_{1}} )
 \otimes_{C} 
 (e_{\alpha_{1}}^{i_{1}} \times \bar{e}^{i_{2}}_{\alpha_{2}} )  \otimes_{C}  ... \otimes_{C}  (e_{\alpha_{p}}^{i_{p}} \times \bar{e}^{i_{0}}_{\alpha_{0}} ), 
\end{multline}
satisfies
\begin{equation}
|| SK || = || K||
\end{equation}
and hence, it is a continuous operator.
\end{proposition}   

\begin{proof} 
The proof goes along the same lines as in the proof of Proposition 33. 
\par
We have
\begin{gather}     
(SK)_{ \;i_{0},\;I,\;\;i_{1}, .. ,\;\;i_{p},   \;\; \;\;I \;\;\;i_{1}, \;\; i_{2},..,\;\;i_{p},\;i_{0} }
^{ \alpha_{0}, \alpha_{0},\alpha_{1},...,\alpha_{p},  \;,\alpha_{0},\;\alpha_{1}, \alpha_{2},..., \alpha_{p},\alpha_{0} } = 
              K_{ \;i_{0},\; i_{1}, ..., \;i_{p}, \;\;\;i_{1},...,\; i_{p},\;i_{0} }
^{ \alpha_{0},\alpha_{1},...,\alpha_{p}, \;\;\alpha_{1}, ..., \beta_{p}, \alpha_{0} } .
\end{gather}   
 The formula (78) tells that the passage from the components of $K$ to the components of $SK$ involves a shift (given by the addition of two columns $[I, \alpha_{0}] $ $[I, \alpha_{0}]$ placed on the second and third positions.  All other components, which are not obtained by this procedure, are equal to zero. The shift prevents non-trivial linear combinations between the components of $K$ to produce the components of $SK$.  These arguments complete the proof of the proposition.
\end{proof}   

\section{Topological Considerations}    

\subsection{Alexander-Spanier Co-homology.}   
To simplify the presentation of the construction of the Alexander-Spanier co-homology $H^{AS}_{\ast}(X, G)$, we assume the spaces 
$X$  are countable, locally finite simplicial complexes and $G$ is an arbitrary commutative group.
\par
Let $C_{AS}^{p}(X, G) := \{  f \;|\; f: X^{p+1} \longrightarrow G \}$, where $f$ is an arbitrary function. Let us call any such function $f$ a
{\em non-localised}  Alexander-Spanier  $p$ co-chain on $X$ with coefficients in $G$.
\par
The Alexander-Spanier co-boundary of $f$, $df$, is the $p+1$ co-chain given by the formula
\begin{equation}
df(x_{0}, x_{1}, ..., x_{p}, x_{p+1}) ;= \sum_{k=0}^{k=p+1} (-1)^{k}f (x_{0}, x_{1}, ..., \hat{x}_{k}, ... , x_{p+1}). 
\end{equation}
\par
The complex $\{ C_{AS}^{\ast}(X, G),  d\}_{\ast} $ is not interesting because it is {\em acyclic}. Formally, the acyclicity of this complex is provided by the homotopy operator $\tilde{s}$
\begin{equation}
(\tilde{s}f) (x_{0}, x_{1}, ..., x_{p-1}) := f (P, x_{0}, x_{1}, ..., x_{p-1}),
\end{equation}
where $P$ is an arbitrarily chosen point in $X$. 
\par
This fact may be immediately understood if we agree to think of the point $(x_{0}, x_{1}, ..., x_{p}) $ as representing the simplex   $[x_{0}, x_{1}, ..., x_{p}]$ having as vertices the {\em arbitrary} points $x_{0}, x_{1}, ..., x_{p} $ of $X$. The set of all such simplexes form a simplicial complex $\tilde{X}$ in which every point of $X$ is a vertex and any such vertices are allowed to be connected to form a simplex of this simplicial complex. Clearly, the simplicial complex $\tilde{X}$ is an infinite (if $X$ is an infinite set) dimensional simplex whose vertices are all points of $X$.
\par
The simplicial complex $\tilde{X}$ is, homotopically, a point. The point $P$ in the construction of the homotopy operator $\tilde{s}$ becomes the vertex of the cone over the simplicial complex $\tilde{X}$.
\par
On the other side, obviously, any non-localised Alexander-Spanier $p$ co-chain $f$ is a simplicial co-chain of the simplicial complex 
$\tilde{X}$ with coefficients in $G$. The Alexander-Spanier co-boundary $df$ is nothing but the usual coboundary of this simplicial co-chain.
\par
The basic idea of the Alexander-Spanier co-homology is to consider a {\em sub}-complex $\tilde{X}^{loc}$
of the simplicial complex $\tilde{X}$. A simplex $[x_{0}, x_{1}, ..., x_{p}]$ will be allowed to belong to this sub-complex provided the points $x_{0}, x_{1}, ..., x_{p}$ are sufficiently close one to each other, i.e. iff they belong to a tubular neighbourhood $\mathit{U}_{p+1}$ of the main diagonal $\nabla_{X}^{p+1} := \{(x, x, ... ,x) \; | \; for \; any \; x \in X \} \subset X^{p+1}$, for any $p$. 
Such a sub-complex will be denoted by $\tilde{X}_{\mathit{U}}$.  
Let $\tilde{U}$ be the collection of such neighbourhoods. 
\par
We assume that tubular neighbourhoods $\mathit{U}_{p+1}$ are compatible, i.e., that by removing any component point $x_{k}$ of a point in $\mathit{U}_{p+1}$ one gets a point in $\mathit{U}_{p}$. This could be realised, for example, by choosing a distance function on $X$ and allow  the points $x_{0}, x_{1}, ..., x_{p}$ to belong to $\mathit{U}_{p+1}$ provided their mutual distances do not exceed a sufficiently small fixed number $0 < \epsilon$.
\par
The sub-complex $\tilde{X}_{\mathit{U}}$ has $X$ as a deformation retract. Therefore the sub-complexes $\tilde{X}_{\mathit{U}}$ are homotopically equivalent to the original simplicial complex $X$ and hence they have isomorphic simplicial co-homologies.
\par
Let us denote the simplicial co-chain complex  associated to the sub-complex $\tilde{X}_{\mathit{U}}$ by 
$\{ C_{AS}^ {\mathit{U}, {\ast}} (X, G) , \; d\}_{\ast} $. Therefore, its homology is isomorphic to the simplicial homology of the complex $X$.
\begin{definition}   
Any co-chain 
$f \in C_{AS}^ {\mathit{U}, p} (X, G) $ 
will be called $\mathit{U}$-{\em local} Alexander-Spanier co-chain.
\end{definition}  

\par
The tubular neighbourhoods $\mathit{U}$  form an inductive system by declaring
$\mathit{U} \preccurlyeq    \mathit{V}$ iff  $\mathit{V}    \subseteqq  \mathit{U}$.  
\par
Let $ C_{AS}^{\ast}(X, G)$ denote the projective limit of the complexes    $\{ C_{AS}^ {\mathit{U}, \ast} (X, G) , \; d\} $
with respect to this filtration.

\begin{definition}   
Any co-chain belonging to the complex 
$ C_{AS}^{p}(X, G)$ 
is called an Alexander-Spanier co-chain of degree $p$.
\end{definition}   

As all cohomology complexes $C_{AS}^{p}(\mathit{U}, G)$ have isomorphic homologies, compatible with the deformation retractions discussed above, one gets the
\begin{theorem} (Alexander-Spanier)  
For any countable, locally finite simplicial complex $X$, the homology of the Alexander-Spanier complex $ C_{AS}^{\ast}(X, G)$ is canonically isomorphic to the simplicial co-homology $H^{\ast}(X, G)$.
\end{theorem}  
If the space $X$ does not have an additional analytical structure (differential structure, e. g.), any  Alexander-Spanier co-homology {\em class}  is the homology class $[f]$ of a function $f \in  C_{AS}^ {\mathit{U}, p} (X, G)$. Its support may be chosen  in an arbitrarily small neighbourhood $\mathit{U}$ of the diagonal.
\par
However, if $X$ does possess a differentiable structure, then already at the level of co-chains $f$, it is possible to associate with $f$ a de Rham differential form $DR(f)$, see e.g. \cite{Connes-Moscovici}
\begin{equation}  
DR(f) :=  \frac{\partial^{p}f(x_{0}, x_{1}, ... ,x_{p} ) }{
\partial x_{1}^{i_{1}},        \partial x_{2}^{i_{2}},  , ... ,  \partial x_{p}^{i_{p}}}{| \nabla_{p}} \;
dx_{1}^{i_{1}} \wedge    dx_{2}^{i_{2}} \wedge,  , ... , \wedge  dx_{p}^{i_{p}}.
\end{equation}   
\begin{theorem}   
The mappings $DR$ form a co-chain homomorphism which induces an isomorphism from the Alexander-Spanier co-homology to the de Rham cohomology of $X$.
\end{theorem}   
\begin{remark}  
If the space $X$ possesses a sufficiently regular differentiable structure (at least $C^{2}$), then any Alexander-Spanier co-homology class $[f]$ may be represented by a closed differential form $DR(f)$ by {\em going to the diagonal}, where this differential form lives,  and therefore, entering the  classical differential geometry.
\par
If the space $X$ does not possess a sufficiently regular analytical structure (as e.g. a $C^{2}$-differentiable structure),  then  Alexander-Spanier co-homology classes of degree $p$ may not be represented by classical differential forms. They can, however,  be represented by  {\em arbitrary} functions (even not continuous), with support in an {\em arbitrarily small} tubular neighbourhood of the diagonal $\nabla_{p+1}$. These are examples of elements in the {\em non-commutative geometry} of $X$.
\par
If the space $X$ possesses intermediate analytical structures, e.g. a Lipschitz structure, one may consider the Whitney complex consisting of  measurable, bounded differential forms, with measurable, bounded exterior derivatives, see e.g. \cite{Teleman_IHES}. If the manifold $X$ has a quasi-conformal structure, then the manifold possesses differential forms whose components satisfy weaker conditions, see e.g. \cite{Donaldson-Sullivan}, \cite{Connes-Sullivan-Teleman}.  Both, in the Lipschitz and quasi-conformal case, the corresponding differential forms live in the classical differential geometry, i.e. they live on the diagonals $\nabla$, although their components are defined almost everywhere and are discontinuous. In these two situations, the $DR$-homomorphisms are defined {\em provided} the corresponding Alexander-Spanier co-chains $f$ belong to the {\em projective} tensor products of the algebra of Lipschitz, resp. the Royden algebra, see Connes-Sulivan-Teleman
\cite{Connes-Sullivan-Teleman} for more details.
\end{remark}   


\subsection{Alexander-Spanier Homology.}  

The Alexander-Spanier homology is dual to the Alexander-Spanier co-homology.
\par
Formally, an Alexander-Spanier chain $\gamma$ of degree $p$ with coefficients in the commutative group $G$ would be 
an {\em infinite} formal sum
\begin{equation}   
\gamma = \sum_{[x_{0}, x_{1}, ... ,x_{p} ]}  \gamma_{[x_{0}, x_{1}, ... ,x_{p} ]} [x_{0}, x_{1}, ... ,x_{p} ],
\end{equation}  
where $\gamma_{[x_{0}, x{1}, ... ,x_{p} ]} \in G.$  
\par
The Alexander-Spanier boundary of $\gamma$ would be 
\begin{equation}   
\partial \gamma = \sum_{k=0}^{k=p} (-1)^{k}
\sum_{[x_{0}, x_{1}, ... ,x_{p} ]}  
\gamma_{[x_{0}, x{1}, ... ,x_{p} ]} [x_{0}, ..., \hat{x_{k}}, ... ,x_{p} ].
\end{equation}  
Although infinite chains (82) could be considered in a homology theory, if $X$ is an infinite set, the boundary $\partial \gamma$ does not make sense because the coefficient of the simplex $ [x_{0}, ..., \hat{x_{k}}, ... ,x_{p} ]$ would be given by an infinite sum with respect of $x_{k} \in X$.
\par
The meaning of the formula (83) could be recovered provided the infinite sum would be replaced by an integral.
\begin{definition}   
A real/complex Alexander-Spanier chain $\gamma$ of degree $p$ on the space $X$ is a real/complex valued Lebesque integrable function $\gamma$ on $X^{p+1}$ with support in a tubular neighbourhood of the diagonal
$\nabla_{X} \in X^{p+1}$.
\end{definition}   
\par
Let $C^{AS}_{p}(X)$ denote the set of all Alexander-Spanier $p$-chains on $X$.
\par
The Alexander-Spanier $k^{th}$ boundary face $\partial_{(p)k}^{AS} \gamma$ of  $\gamma$  is
\begin{equation}   
( \partial_{(p)k}^{AS} \gamma )  (x_{0}, ... ,x_{p-1}) \; := \; 
\int_{X} \gamma   (x_{0}, ... ,x_{k-1}, \;t\;,   x_{k}, ... ,x_{p} )  d \mu(t)
\end{equation}  
\par
The Alexander-Spanier boundary of the chain $\gamma$ is
\begin{equation}   
\partial_{(p)}^{AS} \gamma \;:= \; \sum_{k=0}^{k=p} \; (-1)^{k} \partial^{AS}_{(p)k} \gamma.
\end{equation}  
\par
The directed system of neighbourhoods of the diagonal $\mathit{U}$ discussed in the previous sub-section \S 8.1.
lead to a projective system of complexes
 $\{  C^{AS}_{\ast}, \; \partial^{AS}_{\ast}    \}_{\ast} $ . 
\begin{definition}   
A real/complex Alexander-Spanier chain on the space $X$ is by definition any element of the
complex  $\{  C^{AS}_{\ast}, \; \partial^{AS}_{\ast}    \}_{\ast} $.
\end{definition}   
\begin{theorem} (Alexander-Spanier)  
The homology of the real/complex Alexander-Spanier complex of $X$,  
$\{  C^{AS}_{\ast}, \; \partial^{AS}_{\ast}    \}_{\ast} $, 
 is isomorphic to the singular real/complex 
homology of $X$.
\end{theorem}  

\begin{proposition}   
Let $X$ be a countable, locally finite simplicial complex of dimension $n$ and let $\{\Delta_{\alpha}\}_{\alpha}$ denote its
$n$-dimensional simplices. Suppose $\mu$ is a Lebesque measure on $X$ such that each simplex 
$\Delta_{\alpha}$ has measure $1$
\begin{equation}
\int_{\Delta_{\alpha}}  \;  1 \; d\mu  = 1.
\end{equation} 
Let $G$ be an arbitrary Abelian group and let $C^{AS}_{(p)} (X, G)$ be the set of all $G$-valued Alexander-Spanier $p$-chains on $X$ which are constant on each poly-top  
$\Delta_{\alpha_{0}} \times \Delta_{\alpha_{1}} \times ... \times \Delta_{\alpha_{p}}$.
\par
Then  $ \{
C^{AS}_{(\ast)} (X, G) , \partial^{AS}
\} _{\ast}$
is a complex and its homology is the singular (simplicial) homology $H_{\ast} (X, G)$.
\end{proposition}   
\begin{proof}   
The Alexander-Spanier homology is a homology functor. The five-lemma applied onto the inclusion of the CW-homology complex of $X$ into the Alexander-Spanier complex completes the argument.
\end{proof}    
\begin{remark}
The condition (86) is not important for the overall homology result and may be easily removed accordingly.
\end{remark}
\subsection{Isomorphism between    
$ \{ \; C^{I, \mathit{loc} }_{\ast}, \; b \;  \}_{\ast} $ 
and   
$ \{ C^{AS}_{(\ast)} (X, G) , \partial^{AS} \}_{\ast}  $,  $G = \mathit R$, or $ \mathit C$.
}
\par
\begin{definition}   
The space of {\em local} chains in $ C^{I}_{(p)(HS)}$ is by definition
 $ C^{I, \mathit{loc}}_{(p)(HS)}$ consisting of all chains $ K \in C^{I}_{(p)(HS)}$
whose supports lay in a tubular neighbourhood of the diagonal 
$\triangledown_{X} \in X^{p+1}$.
\end{definition}    
\par
Recall that any $K \in C^{I, \mathit{loc}}_{(p)}(HS(X)) $ has the expression 

\begin{equation}    
K = \sum_{\alpha_{0}, \alpha_{1}, ..., \alpha_{p} }  K^{\alpha_{0}, \alpha_{1}, ..., \alpha_{p} }
  (e^{I}_{\alpha_{0}}  \times \bar{e}^{I}_{\alpha_{1}} )  
\otimes_{\mathit{C}}  (e^{I}_{\alpha_{1}}  \times \bar{e}^{I}_{\alpha_{2}} ) 
\otimes_{\mathit{C}}  .... 
\otimes_{\mathit{C}}  (e^{I}_{\alpha_{p}}  \times \bar{e}^{I}_{\alpha_{0}} ), 
\end{equation}   
where $K^{\alpha_{0}, \alpha_{1}, ..., \alpha_{p} }$ are real or complex numbers.
\par
The corresponding chain $[K] $  is given by formula  (37),  in which all indices $i_{k} = I$ 
\begin{equation}  
[K] = \sum_{\alpha_{0}, \alpha_{1}, ..., \alpha_{p} }  K^{\alpha_{0}, \alpha_{1}, ..., \alpha_{p} }
  [ {\alpha_{1}}, {\alpha_{2}}, ...,  {\alpha_{p}}, {\alpha_{0}} ], 
\end{equation}  
 and the repeated index $I$ was omitted. 
\par
Recall that the formula (40) states
\begin{equation}  
[b K] \;=\; \partial [K].
\end{equation}  
\par
Now we are going to interpret the simplicial chain $[K]$ as an Alexander-Spanier $p$-chain 
$\{  K \}  \in  C_{p}^{AS}(X, G) $ ( $ G = \mathit{R}, \mathit{C}$). To do this, we agree to think of this function as taking the constant value  $K^{\alpha_{0}, \alpha_{1}, ..., \alpha_{p} }$ on the poly-top
$  \Delta_ {\alpha_{1}} \times \Delta_{\alpha_{2}} \times ...\times \Delta_ {\alpha_{p}} \times \Delta_ {\alpha_{0}}  $.
\par

\begin{theorem}    
-i) Suppose each simplex $\Delta_{\alpha}$ has measure 1, see formula (86). Then
\begin{equation}   
\partial^{AS}_{(p)k} \{K\}= \partial_{(p)k} [K],  \hspace{0.2cm} for \; any \;    0 \leq k \leq p.
\end{equation}    
Therefore,
\begin{equation}   
\partial^{AS} \{K\} = \partial [K].
\end{equation}   
\par
-ii) If the supports of the chains of $ \{ \; C^{I, \mathit{loc} }_{\ast}, \; b \;  \}_{\ast} $ lay in a tubular neighbourhood of the diagonals
$\nabla_{p}$, for any $p$, then the homology of the complex $ \{ \; C^{I, \mathit{loc} }_{\ast}, \; b \;  \}_{\ast} $ is isomorphic to the Alexander-Spanier homology $H^{AS}_{\ast} (X)$.
\end{theorem}  
\begin{proof}  
-i) It is sufficient to prove formula (85). One has 
\begin{multline}  
\partial^{AS}_{(p)k} \{K\} (x_{0}, x_{1}, ..., x_{p-1}) := \int_{X} \{K\} (x_{0}, x_{1}, ..., x_{k-1}, \;t, \; x_{k},.., x_{p-1}) \; d\mu(t) = \\
=       \sum_{\alpha_{0}, \alpha_{1}, ..., \alpha_{p}}
K^{\alpha_{0}, \alpha_{1} ... \alpha_{p} }   \int_{X}  
\chi_{\alpha_{0}}(x_{0}) ... 
\chi_{\alpha_{k-1}}(x_{k-1})  
 \; \chi_{\alpha_{k}}(t)  \;
\chi_{\alpha_{k+1}}(x_{k}) ...
\chi_{\alpha_{p}}(x_{p-1})  
 \; d\mu(t) = \\
 =   \; \sum_{\alpha_{0}, \alpha_{1}, ..., \hat{\alpha}_{k},...,\alpha_{p}} \; \sum_{\alpha_{k}}
K^{\alpha_{0}, \alpha_{1} ... \alpha_{p} }    
\chi_{\alpha_{0}}(x_{0}) ... 
\chi_{\alpha_{k-1}}(x_{k-1})  
  \;
\chi_{\alpha_{k+1}}(x_{k}) ...
\chi_{\alpha_{p}}(x_{p-1})  =
\end{multline}
\begin{equation*}
= \partial_{(p)k} \{K\} (x_{0}, x_{1}, ..., x_{p-1}).
 \end{equation*}   
where  $\chi_{\Delta}$ are characteristic functions.
\par
Part -i) implies -ii).
\end{proof}  

\section{Control of the Supports of Hochschild Chains.}   
In this section we are going to show that if the triangulation of the space $X$ is sufficiently fine, then the quasi-isomorphisms treated in \S 5 and \S 6
\begin{equation} 
\{  \;   C_{\ast}(HS) , b  \; \} \longleftarrow \{  \;  C^{\Delta}_{\ast}(HS) , b  \; \}  \longleftarrow  \{ \;  C^{I}_{\ast}(HS) , b \; \}
\end{equation}  
pass to the {\em local} sub-complexes. The computation of the homology of the sub-complex 
 $ \{ \;  C^{I, \mathit {loc}}_{\ast}(HS) , b \; \}$ 
 will complete the proof of Theorem 1.
 \par
We remark that the Hochschild boundary $b$ as well as the homotopy operators $s$ and $S$ discussed in \S 6.1. and \S 6.2. send continuous Hochschild chains which have  {\em small support} about the diagonal into chains of the same type. For the purpose of controlling the support of Hochschild chains we introduce the {\em diameter} of chains (see Definition 48 below). 
\par
Assume for simplicity that $X$ is connected.
We say that the {\em simplicial distance} between the top dimensional simplexes  $\Delta_{\alpha}$ and $\Delta_{\beta}$ is $d$ provided $d$ is the minimum number of $1$-simplices needed to connect a vertex of the simplex $\Delta_{\alpha}$ with a vertex of the simplex $\Delta_{\beta}$.
 \par
 The simplicial distance leads to an increasing filtration of the space of elementary chains.
\begin{definition}   
 Given the elementary chain
 \begin{equation*}      
K =  (e^{i_{0}}_{\alpha_{0}}  \times \bar{e}^{j_{0}}_{\beta_{0}} )                
                            \otimes_{\mathit{C}}  ...   \otimes_{\mathit{C}}                            
 (e^{i_{p}}_{\alpha_{p}}  \times \bar{e}^{j_{p}}_{\beta_{p}} ).
\end{equation*}   
we define its {\em diameter}
\begin{equation}   
\mathit{Diam}(K) := \text {maximal simplicial distance between the simplices}  \;
\Delta_{\alpha_{0}},..., \Delta_{\alpha_{p}},  \Delta_{\beta_{0}},...,  \Delta_{\beta_{p}}
\end{equation} 
\end{definition}   
\begin{definition}  
For any natural number $N$, let 
$C_{p}(HS(X))_{N}  \subset C_{p}(HS(X))$ be
\begin{equation}   
C_{p}(HS(X))_{N} :=  \{ K \;|\; K = \sum _{\alpha_{0},\beta_{0},...,\alpha_{p},\beta_{p}}^{i_{0}, j_{0},...,i_{p}, j_{p}}
K^{\alpha_{0},\beta_{0},...,\alpha_{p},\beta_{p}}_{i_{0}, j_{0},...,i_{p}, j_{p}}
(e^{i_{0}}_{\alpha_{0}}  \times \bar{e}^{j_{0}}_{\beta_{0}} )                
                            \otimes_{\mathit{C}}  ...   \otimes_{\mathit{C}}                            
 (e^{i_{p}}_{\alpha_{p}}  \times \bar{e}^{j_{p}}_{\beta_{p}} ), 
 \end{equation}
 \begin{equation*}
 \mathit{Diam} (e^{i_{0}}_{\alpha_{0}}  \times \bar{e}^{j_{0}}_{\beta_{0}} )                
                            \otimes_{\mathit{C}}  ...   \otimes_{\mathit{C}}                            
 (e^{i_{p}}_{\alpha_{p}}  \times \bar{e}^{j_{p}}_{\beta_{p}} ) \leq N.
\}
\end{equation*} 
A chain belonging to $C_{p}(HS(X))_{N}$ will be called {\em $N$-local}.
\end{definition}  
 \par
In particular, the simplicial distance leads to an increasing filtration of the space of Hilbert-Schmidt operators 
$\mathit{HS} (X)_{N}$ 
\begin{equation}   
\mathit{HS} (X)_{N} \; :=\; \sum_{\alpha \beta, ij} K_{ij}^{\alpha \beta} \; (e_{\alpha}^{i} \times \bar{e}^{j}_{\beta} ), \hspace{0.5cm}
    \mathit{Diam}(e_{\alpha}^{i} \times \bar{e}^{j}_{\beta} ) 
    \leq N.
\end{equation}    
\par
It is clear that
\begin{equation}   
\mathit{HS} (X)_{m} \circ \mathit{HS} (X)_{n}  \subset \mathit{HS}(X)_{m+n+1}.
\end{equation}  
\begin{definition}  
Let $N$ be any natural number such that the supports of the chains of the complex $C_{\ast} (HS)_{N}$ lay in a tubular neighbourhood
of the diagonals $\nabla_{p}$, for any $p$.  Define
\begin{gather}
C^{\Delta}_{\ast}(HS)_{N} := C^{\Delta}_{\ast}(HS) \cap C_{\ast} (HS)_{N} \\   
C^{I}_{\ast}(HS)_{N} := C^{I}_{\ast}(HS) \cap C_{\ast} (HS)_{N}.  
\end{gather}
\end{definition}  
\begin{proposition}  
The Hochschild boundary operator $b$, as well as the homotopy operators $s$ and $S$, do not increase the diameter of the chains. 
\par
Therefore, the homotopy operators $s$, $S$ are well defined in the corresponding complexes $C{\ast}(HS)_{N}$,  $C^{\Delta}_{\ast}(HS)_{N}$,  $C^{I}_{\ast}(HS)_{N}$.
\end{proposition}  

\begin{theorem}  
Let $N$ be any natural number with the property that the supports of the chains of the complex $C_{\ast} (HS)_{N}$ lay in a tubular neighbourhoodof the diagonals $\nabla_{p}$, for any $p$.  Then
\begin{equation}
\{  \;   C_{\ast}(HS)_{N} , b  \; \} \longleftarrow \{  \;  C^{\Delta}_{\ast}(HS)_{N} , b  \; \}  \longleftarrow  \{ \;  C^{I}_{\ast}(HS)_{N} , b \; \}
\longrightarrow \{  \;   C_{\ast}^{AS} (X)_{N} , \partial  \; \}
\end{equation}
are quasi-isomorphisms.
\end{theorem}  
\begin{proof}  
The notion of gap passes to the {\em local} spaces $C_{\ast}(HS)_{N}$.  Therefore, the splitting stated by Proposition 7 continues to hold in these spaces.
\par
The operator $s$ continues to satisfy the identity (22), so that Proposition 8 holds inside the complex 
$\{  \;   C_{\ast}(HS)_{N} , b  \; \}$ too. These prove that the first arrow of (100) is a quasi-isomorphism.
\par
The $S$ continues to satisfy the identity (33) of Proposition 13 inside the complex $\{  \;  C^{\Delta}_{\ast}(HS)_{N} , b  \; \}$.
The algebraic modifications discussed in the proof of Proposition 27 may be carried out inside the sub-complex
$\{ \;  C^{I}_{\ast}(HS)_{N} , b \; \}$. Therefore, the second arrow of the relation (100) is a quasi-isomorphism.
\par 
The discussion made in \S 8.3. passes to sub-complexes $\{ \;  C^{I}_{\ast}(HS)_{N} , b \; \}$,
$ \{  \;   C_{\ast}^{AS} (X)_{N} , \partial  \; \} $. Theorem 47 remains valid in these contexts.
These prove that the last arrow of the relation (100) is a quasi-isomorphism. This completes the proof of the theorem.
\end{proof}  
Theorem 52. implies the following result
\begin{theorem}   
Let $HS(X)$ denote the algebra of real/complex Hilbert-Schmidt operators on the countable, locally finite, homogeneous, simplicial complex $X$ of dimension $N$.
\par
Let $N$ be any natural number with the property that the supports of the chains of the complex $C_{\ast} (HS)_{N}$ lay in a tubular neighbourhoodof the diagonals $\nabla_{p}$, for any $p$.  
\par
Then the homology of the Hochschild sub-complex
\begin{equation}
\{  \;   C_{\ast}(HS)_{N} , b  \; \}  
\end{equation}
is isomorphic to the real/complex Alexander-Spanier homology of $X$
\begin{equation}
H_{\ast}^{AS} (X)
\end{equation}
\end{theorem}  

\section{{\em Local} Hochschild Homology of the Algebra of Hilbert-Schmidt Operators.}  

\subsection{Preliminaries.}
We discuss here the {\em local} Hochschild homology of the algebra of Hilbert-Schmidt operators on the simplicial complex $X$,  
see  \S 4.1. Recall that {\em local} Hochschild and cyclic homology were introduced  in \cite{Teleman_arXiv} as a tool to set up the index theorem in a more natural environment. To begin with, for this reason, we are interested in discribing this notion primarily on the algebra of {\em quasi-local, bounded operators on a Hilbert space of $L_{2}$-sections in vector bundles}. This class of operators contains  pseudo-diffeential operators and integral operators.
\par
To start this argument we assume that the $n$-dimensional simplicial complex $X$ is embedded in an Euclidean space. 
Let $\mathit{r}$ be the induced metric on $X$.
\par
As explained in \S 4.1.,  any Hilbert-Schmidt operator on $X$ is defined by its kernel $K: X \times X \longrightarrow \mathit{C}$.
The space of Hilbert-Schmidt operators on $X$ is filtrated by the support if its elements. We have seen in \S 7.1. that the description of the elements of the  Hochschild complex over this algebra involves real/complex valued  $L_{2}$-functions over $(X \times X)^{p+1}$. 
In \S 9. we introduced a filtration of the Hochschild chains based on the simplicial structure of $X$. This filtration has its own interest.
\par
In this section we are going to introduce a new filtration based on the size of the support measured in terms of {\em distance to the diagonal}.
\begin{definition}  
Let $P = (x_{0} ,y_{0}) | x_{1} ,y_{1} | ,...,| x_{p}, y_{p}) \in (X \times X)^{p+1}$. Define the {\em distance from the point P to the main diagonal} $\nabla_{X}^{2(p+1)}$ by
\begin{equation}   
\mathit{r} ((x_{0} ,y_{0}), (x_{1} ,y_{1}), ...,(x_{p}, y_{p}), \; \nabla_{X}^{p+1}\; ) := Max ^{k=p}_{k=0} \{ 
 \mathit{r} (x_{k}, y_{k}),  \mathit{r} (y_{k}, x_{k+1}) \} .   \hspace{0.2cm}  ( x_{p+1} := x_{0})
\end{equation}  
Let $\epsilon$ be any positive number.  Let 
\begin{equation}   
U^{2(p+1)}_{\epsilon} := \{  P \;|\; P \in (X \times X)^{p+1},      \mathit{r} (P)  < \epsilon            \}.
\end{equation}   
\end{definition}   
 \par
 Denote by $HS(X)_{\epsilon}$ the set of all Hilbert-Schmidt operators whose kernel have support in   
 $U^{2(p+1)}_{\epsilon} $.  The main idea of {\em local} Hochschild homology is to consider the homology of the sub-complex of the Hochschild complex  consisting of chains which have small support about the diagonal.
 For chains of degree zero we intend elements of $HS(X)_{\epsilon}$ with $\epsilon$ small. Smoothing operators with arbitrarily small support appear in the Connes-Moscovici local index theorem \cite{Connes-Moscovici} and they hold the topological information relating to the index formula.
 \par
 Looking for chains of degree $p$ which have small support about the diagonal $\nabla^{2(p+1)}_{X}$, we start by considering elements of 
$\otimes^{p+1}_{\mathit{C}} HS(X)_{\epsilon}$. For any element in this set, for any point in its support,  one has 
 $\mathit{r}(x_{k}, y_{k}) < \epsilon$, for any $0 \leq k \leq p$. For the elements in this set there is, so far, no condition on the distances  $\mathit{r}(y_{k}, x_{k+1})$. 
 The support (Supp) of $K_{0} \otimes K_{1} \otimes ,...,\otimes K_{p}$ is 
 $Supp(K_{0}) \times Supp(K_{1})\times,..., \times Supp(K_{p}) $. To insure that the support of this element
 is small about the diagonal $\nabla_{X}^{2(p+1)} := \{ (x,x,...,x) \in (X \times X)^{p+1}  \}$,  we impose also the conditions $\mathit{r}(y_{k}, x_{k+1}) < \epsilon$, for any  $0 \leq k \leq p$. Now we are in the position to define
 precisely the $\epsilon$-{\em local} Hochschild chains in the completed Hochschild complex.
 \begin{definition}    
 An $\epsilon$-{\em local} Hochschild chain of degree $p$ is by definition an element of
 \begin{equation}  
 C_{p}(HS(X))_{\epsilon} := \{   K \;|\; K \in C_{p} (HS(X)), \; Supp (K) \subset U^{2(p+1)}_{ \epsilon }     \}.
 \end{equation}  
 \end{definition}   
\par
Multiplying Hilbert-Schmidt operators increases their support. For any natural numbers $k_{1}$,  $k_{2}$
 \begin{equation}   
 HS(X)_{k_{1}\epsilon}  \circ HS(X)_{k_{2}\epsilon}  \subset HS(X)_{(k_{1} + k_{2}) \epsilon}.
 \end{equation}   
Given that the multiplication of Hilbert-Schmidt operators increases the support (106), the vector spaces  
$C_{p}(HS(X))_{\epsilon} $  satisfy
\begin{equation}   
b \; C_{p}(HS(X))_{\epsilon}  \subset C_{p}(HS(X))_{2\epsilon} .
\end{equation}   
To simplify the notation, we write $C_{p}(HS(X))_{\epsilon} = C_{p, \epsilon}$

\begin{definition}   
For any $0 < \epsilon $ we define
\begin{equation}   
H^{\epsilon}_{p} (HS(X)) := \end{equation}
\begin{equation*}
 = \frac
 { Ker \; b: C_{p,\epsilon} \rightarrow C_{p, 2\epsilon} }{ \{Im \; b: C_{p+1, k\epsilon} \rightarrow C_{p, 2k\epsilon} \}
 \cap \{        Ker \; b: C_{p,\epsilon} \rightarrow C_{p, 2\epsilon}      \}
 },  \; \;    2  \leq k, 
\end{equation*}     
where $k$ is fixed number.
\end{definition} 
For any $\epsilon'  < \epsilon$ one has $ C_{p,\epsilon'}  \subset C_{p,\epsilon} $ and therefore there is an induced mapping in homology
$\Delta(\epsilon, \epsilon'): H^{\epsilon'}_{p} (HS(X))  \rightarrow H^{\epsilon}_{p} (HS(X)) $
\begin{equation}     
\Delta(\epsilon, \epsilon'): \;
  \frac
 { Ker \; b: C_{p,\epsilon'} \rightarrow C_{p, 2\epsilon'} }{ \{Im \; b: C_{p+1, k\epsilon'} \rightarrow C_{p, 2k\epsilon'} \}
 \cap \{        Ker \; b: C_{p,\epsilon'} \rightarrow C_{p, 2\epsilon'}      \}
 } \longrightarrow
\end{equation}     
\begin{equation*}
 \longrightarrow  \frac
 { Ker \; b: C_{p,\epsilon} \rightarrow C_{p, 2\epsilon} }{ \{Im \; b: C_{p+1, k\epsilon} \rightarrow C_{p, 2k\epsilon} \}
 \cap \{        Ker \; b: C_{p,\epsilon} \rightarrow C_{p, 2\epsilon}      \}
 }
\end{equation*}     
\begin{definition}   
The {\em local} Hochschild homology of the algebra of Hilbert-Schmidt operators is given by the formula
\begin{equation}
H^{\mathit{loc}}_{p} (HS(X)) :=  \; \underset{ \epsilon \; \searrow \; 0}{ProjLim}   \;  H^{\epsilon}_{p} (HS(X))
\end{equation}   
\end{definition}   

\subsection{Distance Control of Supports {\em vs.}  Simplicial Control. The Result.}  
\par
Let $K \in C_{p,\epsilon}(HS(X))$,  $0 \leq p \leq n= dim X$.
\par
We intend to look here on the size of the support of $K$ after all algebraic modifications used in \S 5 and \S 6 are performed.
To this purpose notice that the Hochschild boundary $b$ doubles the "diameter" of chains, see (107).
On the  other side, the homtopy operators $s$, (see \S 6.1)  and $S$, (see \S 6.2),  do not modify the size of the support.
\par
We chose an $\epsilon$. We suppose that the diameter of each of the simplices $\Delta_{\alpha}$ is less than $\epsilon$. 
If this condition is not satisfied, we consider a higher order barycentric sub-division. The condition (86) of  Proposition 44  changes by replacing the measure $1$ of each simplex by a constant, depending only on the number of barycentric sub-divisions.  To avoid further complications of the homological picture, we may assume that for each maximal simplex $\Delta_{\alpha}$ the chosen function $I_{\alpha}$ is a normalised constant function, see Definition 22, \S 6.3.  With these precautions taken,  the elementary chains belong 
to $C_{p,\epsilon}(HS(X))$.
\par
For any $K \in K \in C_{p,\epsilon}(HS(X))$ the decomposition (17) of Proposition 7. holds.
\par
Given that the homotopy formulas (22) and (33) involve only once the Hochschild boundary, Lemma 9.  and Proposition 33. hold inside
$C_{p,2 \epsilon}(HS(X))$. The operator $\theta_{p}$ is available in the space $C_{p,2 \epsilon}(HS(X))$. 
\par
Furthermore, the Proposition 27. involves the operator $\theta_{p}$ raised to the maximum power $p+1.$  
\par
To summarise, we may state that all considerations made in the \S 6 hold in the vector spaces
$C_{p,2^{p+2} \epsilon}(HS(X))$. The connection between the {\em local} Hochschild homology and the Alexander-Spanier homology discussed in \S 9 holds provided the supports of the chains sit inside a tubular neighbourhood of the diagonal. 
Now we are in the position to state the
\begin{theorem}    
Let $X$ be any countable, locally finite, homogeneous simplicial complex of dimension $n$. Let $\mathit{r}$ be a distance function, see \S 10.1. 
\par
Let $\epsilon$ be a positive number. Suppose each maximal dimension simplex of $X$ has diameter less than $\epsilon$.
\par
Let $p$ be a natural number. Suppose the simplicial decomposition of $X$ is sufficiently fine so that the set
\begin{equation}  
U_{2^{p+2} \epsilon}^{2(p+1)} \text{\; is a tubular neighbourhood of the diagonal.}
\end{equation}   
Then
\begin{equation}  
H^{\mathit {loc}}_{p} (HS(X)) \text{\; is naturally isomorphic to}  \; H^{AS}_{p} (X).
\end{equation}  
\end{theorem}   
\par
The condition (111) tells us that in order to prove Theorem 1 we need to consider a sufficiently fine decomposition
of the simplicial complex $X$ before the constructions made in \S 5 and \S 6 start. This modification does not change the Hilbert-Schmidt algebra; it only changes the representation of its elements. 
This completes the proof of Theorem 1.


\end{document}